\documentclass[a4paper,11pt,leqno]{amsart}
\usepackage{a4wide}
\usepackage{amsmath}
\usepackage{amsfonts}
\usepackage{amssymb}
\usepackage{dsfont}
\usepackage{bm}
\usepackage{mathrsfs}
\usepackage{graphicx}
\usepackage{fancyhdr}
\usepackage{amsthm}
\usepackage{empheq}
\usepackage{cases}
\usepackage[all]{xy}
\usepackage{stmaryrd}
\usepackage[colorlinks, citecolor=blue, linkcolor=red]{hyperref}
\usepackage{mathrsfs}

\usepackage{mathtools}

\def\RR{{\mathbb R}}

\def\a{\alpha}
\def\b{\beta}
\def\g{\gamma}
\def\d{\delta}
\def\e{\eta}

\def\p{\partial}

\def\R{\bm{R}}
\def\W{\bm{W}}
\def\C{\bm{C}}
\def\D{\bm{D}}
\def\X{\bm{X}}

\def\A{\bm{A}}
\def\E{\bm{E}}
\def\H{\bm{H}}
\def\P{\bm{\mathcal{PF}^E_k}}
\def\Pa{\bm{\mathcal{PF}^E_{k,\alpha}}}
\def\Pm{\bm{\mathcal{PF}^M_k}}
\def\Pma{\bm{\mathcal{PF}^M_{k,\alpha}}}
\def\T{\bm{\mathcal{T}}}
\def\gb{\bm{\gamma}}
\def\gbr{\bm{\bar{\gamma}}}
\def\Gb{\bm{\Gamma}}

\def\L{\mathbf{\Lambda}}

\def\S{\mathcal{S}}

\newtheorem{ackn}{Acknowledgments\!}

\newtheorem{theorem}{\textbf{Theorem}}[section]
\newtheorem{lemma}[theorem]{\textbf{Lemma}}
\newtheorem{proposition}[theorem]{\textbf{Proposition}}

\newtheorem{definition}[theorem]{\textbf{Definition}}
\newtheorem{remark}[theorem]{\textbf{Remark}}

\numberwithin{equation}{section}

\title[A Weyl Entropy of Pure Spacetime Regions]
{A Weyl Entropy of Pure Spacetime Regions}

\author[Francesco Belgiorno]{Francesco Belgiorno}
\address[Francesco Belgiorno]{Dipartimento di Matematica, Politecnico di Milano, Piazza Leonardo da Vinci 32, 20133 Milano, Italy}
\email[]{francesco.belgiorno@polimi.it}

\author[Giovanni Catino]{Giovanni Catino}
\address[Giovanni Catino]{Dipartimento di Matematica, Politecnico di Milano, Piazza Leonardo da Vinci 32, 20133 Milano, Italy}
\email[]{giovanni.catino@polimi.it}

\begin{document}
%\title{On fourth order conformally invariant differential operators on Riemannian manifolds}

%\date{\today}

\begin{abstract} We focus on the Penrose's Weyl Curvature Hypothesis in a general framework encompassing many specific models discussed in literature. We introduce a candidate density for the Weyl entropy in pure spacetime perfect fluid regions and show that it is monotonically increasing in time under very general assumptions. Then we consider the behavior of the Weyl entropy of compact regions, which is shown to be monotone in time as well under suitable hypotheses, and also maximal in correspondence with vacuum static metrics. The minimal entropy case is discussed too.
\end{abstract}

\maketitle

\begin{center}


%\medskip
%
%\centerline{\bf AMS subject classification:  ***}

\end{center}

%\tableofcontents

\

\section{Introduction}

Let $(\X,\gb)$ be a four-dimensional smooth connected Lorentzian manifold with signature $(-,+,+,+)$. We will say that $(\X,\gb)$ is a {\em spacetime} if the metric $\gb$ satisfies the Einstein equation
\begin{equation}\label{eq-ein}
\R_{\a\b}-\frac{1}{2} \R\,\gb_{\a\b} = \T_{\a\b}  
\end{equation}
where $\R_{\a\b}$, $\R$ denote the Ricci and the scalar curvature of $\gb$ and $\T_{\a\b}$ is a symmetric two tensor. The tensor $\T$ is referred as the {stress-energy tensor}. When $\T\equiv 0$ we will say $(\X,\gb)$ is a {\em vacuum spacetime} and the Einstein equation reads
\begin{equation*}
\R_{\a\b}-\frac{1}{2} \R\,\gb_{\a\b} = 0\,
\end{equation*}
or equivalently
\begin{equation}\label{eq-einv}
\R_{\a\b} = 0\,. 
\end{equation}
By the standard decomposition of the curvature tensor $\R_{\a\b\g\d}$ of the metric $\gb$, the geometry of a spacetime is completely determined by its {\em Weyl curvature} $\W_{\a\b\g\d}$ (and the stress-energy tensor). More precisely, the Riemann curvature tensor of a spacetime satisfying \eqref{eq-ein} is given by
\begin{equation}\label{eq-riemdecW}
\R_{\a\b\g\d} = \W_{\a\b\g\d}+\frac12\left(\T_{\a\g}\gb_{\b\d}-\T_{\a\d}\gb_{\b\g}+\T_{\b\d}\gb_{\a\g}-\T_{\b\g}\gb_{\a\d}\right)-\frac{\T}{3}\left(\gb_{\a\g}\gb_{\b\d}-\gb_{\a\d}\gb_{\b\g}\right)\,
\end{equation}
where $\T:=\operatorname{trace}(\T)=\gb^{\a\b}\T_{\a\b}$, and thus it is natural to observe all the geometrical/physical properties of the space time arise from the Weyl and the stress-energy tensors.

The original Weyl Curvature Hypothesis by Roger Penrose \cite{pen3} represents a still unproven conjecture about the entropic contribution of the gravitational field to the overall entropy of the Universe. The naive idea can be expressed as follows: according to Penrose, 
the entropy increase associated with the second law of thermodynamics is hardly compatible with the fact that, at the Big Bang, the Universe itself was in a 
uniform state of thermal equilibrium, which means that entropy was maximal. Penrose's hypothesis is that only the matter field degrees of freedom (dof)  were in equilibrium, whereas the dof associated with the gravitational field were not, and could maintain their very low entropy content for a long time, until the formation of galaxies and stars,  because of the weakness of the gravitational interaction. There is a naive suggestion about the possibility to describe qualitatively the entropy of the gravitational dof, 
which consists in distinguishing two different contributions to the Riemann curvature as given in \eqref{eq-riemdecW}: on the one hand, there is the Ricci contribution, 
which is directly determined by the Einstein equation, and is then related to non-gravitational dof; on the other hand, there is a purely gravitational contribution, which is associated with the Weyl tensor. The Weyl Curvature Hypothesis amounts to hypothesizing that the Weyl curvature is zero at the 
initial singularity (Big Bang), to be compared with a divergent Ricci curvature. The hypothesis in itself does not determine a specific expression for the 
gravitational entropy in terms of the Weyl curvature tensor. There is a number of proposals in the physical literature, which take into account specific models, 
mainly of cosmological nature. Several proposals involve ratios between scalar functions of the Weyl tensor and scalar function of other curvature invariants, like e.g. the Ricci tensor as in the original ansatz by Penrose. All proposal are in general applied to specific models. \cite{goowai1, goowai2, goocolwai, pellak, rotann, waiand, gro}.  
One important contribution in this field is represented by \cite{ell}, where five conditions are prescribed for any 
good definition of gravitational entropy. For completeness, we recall them, with the same notation:
\begin{itemize} 
\item[(E1)] it should be nonnegative; 
\item[(E2)] it should vanish only if the Weyl curvature tensor vanishes; 
\item[(E3)] it should measure the local anisotropy of the free gravitational field; 
\item[(E4)] it should reproduce the Bekenstein-Hawking entropy in the black hole case; 
\item[(E5)] it should increase monotonically as structures form in the Universe. 
\end{itemize}
See also the discussion at the end of this section. We sum up our main definition for our entropic function in what follows. We stress that our Weyl entropy $\S$ (as well as $\S^{\bf{pf}}$) is not yet integrated over a region of space, i.e. is the entropy of an infinitesimal volume (with the correct measure).

\begin{figure}\label{f-1}
\centering
\includegraphics[scale=0.7]{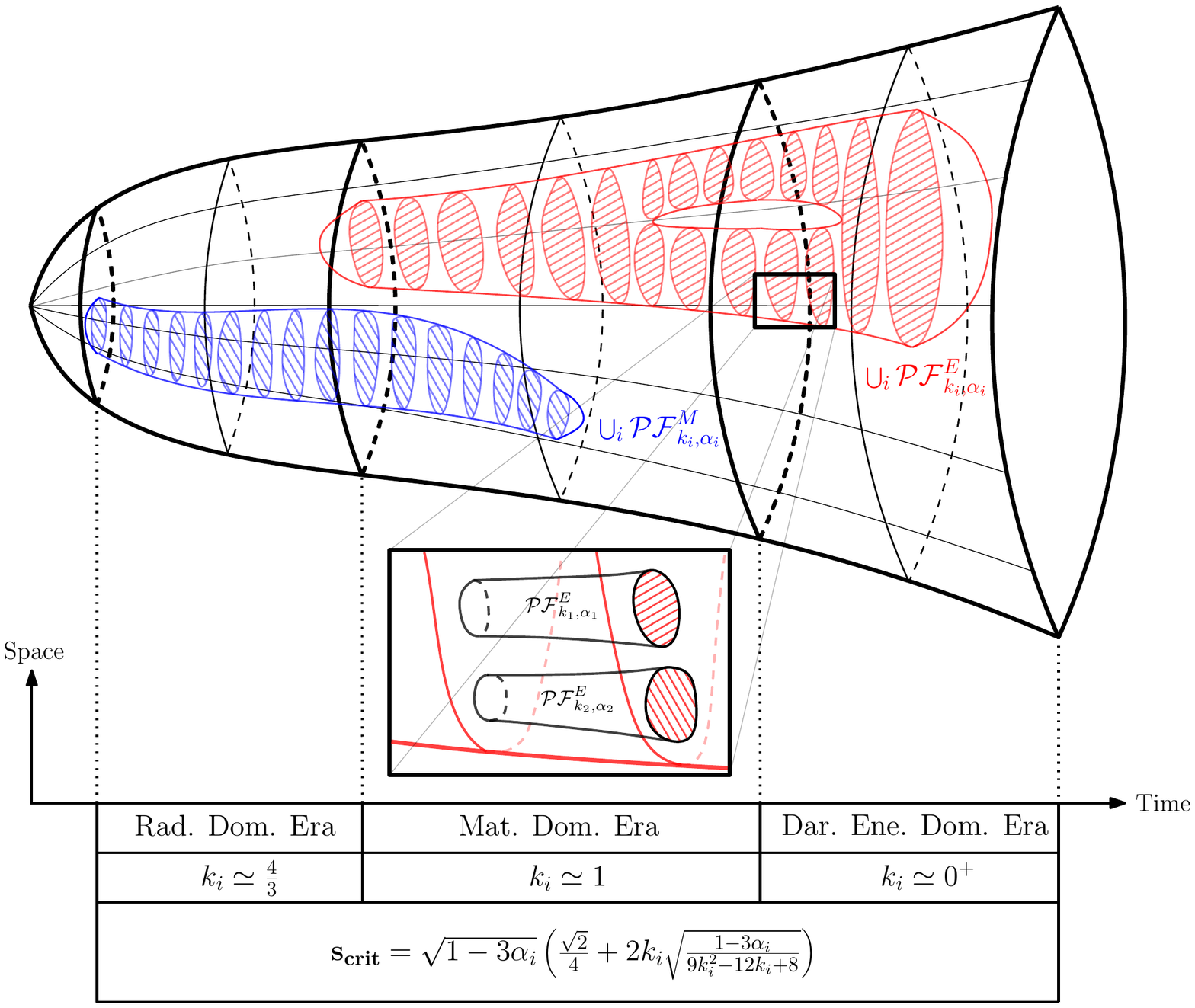}
\caption{A picture of a {\em Weyl Entropic Electric/Magnetic Regions} of spacetime. Each region is given by the union of perfect fluid electric/magnetic regions with given parameters $k_i, \alpha_i$.}
\end{figure}

Consider a globally hyperbolic spacetime  (see Section \ref{s-glo} for details)
$$
(\X,\gb)=(I\times M^3, -N^2 dt^2+g).
$$ 
At given point $(t,x)\in \X$ we define the {\em Weyl entropy} $\S=\S(t,x)$ as
\begin{align*}
\S &:= \frac{|\W_{\a\b\g\d}|_{\gbr}}{|\R_{\a\b\g\d}|_{\gbr}}\sqrt{g}\,,
\end{align*}
where $\R_{\a\b\g\d}\neq 0$ and $\S=1$ where $\R_{\a\b\g\d}=0$. Here $\gbr= N^{2} dt^{2}+g$ is the Riemannian metric associated to $\gb$ and $\sqrt{g}$ denotes the square-root of the determinant of the space metric $g$. In this paper we prove a  monotonicity result for the related Weyl entropy
$$
\S^{\bf{pf}} := \S+ \mathbf{s_{crit}} \sqrt{g}\,,
$$
where $\mathbf{s_{crit}}$ is an explicit nonnegative function depending on the parameters $k,\alpha$ of the perfect fluid electric region $\Pa$. We refer to Section \ref{s-mon} for the precise definitions.

\begin{theorem}\label{t-mone} On every $k$-perfect fluid electric $\alpha$-expanding region $\Pa$ satisfying \eqref{eq-asspar} the Weyl entropy $\S^{\bf{pf}}$ is monotonically increasing in time. 
\end{theorem}

We have a similar result for magnetic regions.

\begin{theorem}\label{t-monm} On every $k$-perfect fluid magnetic $\alpha$-expanding region $\Pma$ the Weyl entropy $\S$ is monotonically increasing in time.
\end{theorem}

We will call a {\em Weyl Entropic Electric/Magnetic Region} of spacetime the union of $k$-perfect fluid electric/magnetic $\alpha$-expanding regions with different parameters $k_i,\alpha_i$ (see Figure 1). By Theorems \ref{t-mone} e \ref{t-monm}, in these regions the Weyl entropy in increasing in time (E5). We also refer to the aforementioned regions as `pure spacetime regions'.

\medskip

If the purely electric/magnetic character of the manifold is missing, there is no clear way to obtain a sensible increasing function of entropic type. 
It is not yet known how to study a mixed situation where the manifold is neither purely electric nor purely magnetic.

\medskip

The parameter $k$ relates the pressure $P$ and the mass density $M$ of the fluid through the equation of state
$$
P=(k-1)M
$$ 
and is standard in cosmology, where it is a constant labelling different epochs in the Universe evolution (see Figure 1).  In the present picture, it is possible to 
allow for a dependence on time and space. A dependence on space-time coordinates can be allowed also for the further parameter $\alpha$, which measures the expansion and the homogeneity in space of the Universe (see Section \ref{s-mon} for the precise definitions).

%\
%
%{\color{red} 
%Dire che valgono le ipotesi di Ellis (E1)-(E5) sia per $\S$ che per $\S^{pf}$. 
%
%By definition, $\S$ is non-negative (E1) and $\S$ vanishes if and only if the Weyl tensor vanishes (E2). 
%
%\
%
%Dire che vale (E2): però ora oltre ad avere Weyl zero, deve anche valere $s_{crit}=0$, cioè $\alpha=1/3$. 
%
%\
%
%Introdurre l'entropia (anche quella modificata( su un dominio $U$ con il fattore Area/Volume. Motivarlo dicendo che è si controlla con l'Area (Bekenstein bound?) 
%
%\
%
%Penso che si possa dimostrare che se Weyl è zero e $\alpha=1/3$ in $I'\times U$ allora lo spazio tempo è localmente FLRW. E' interessante. Cioè l'entropia deformata in $U$ $\S^{\bf{pf}}$ con $s_{crit}$ è minima (zero) se e solo se sei su un modello FLRW.
%
%\
%
%Dire che grazie al fattore Area vale (E4).}
%
%\
%
%{\color{red} Spiegare (E3) o comunque dire due parole su questa ipotesi. Per tutte le altre direi che ci siamo}
%

\medskip

From a physical point of view, there is not yet an unique way for identifying a gravitational entropy $\S$, lacking a universally accepted and definitive quantum gravity theory. Notwithstanding, as discussed in \cite{ell}, one can provide in different ways an ansatz for a putative entropy function for the gravitational field. In particular, the ratio involving the (Riemannian) modulus of the Weyl tensor and the (Riemannian) modulus of the Riemann tensor is interesting because it represents the relative weight of the Weyl curvature contribution with respect to the overall Riemannian curvature contribution, i.e. of the purely gravitational contribution to the curvature with respect to the sum of the gravitational and the matter field contributions, and is a pure number contained in the interval $[0,1]$. There is not yet a specific statistical mechanical suggestion for such a specific ansatz, whose reliability can be judged only a posteriori. Furthermore, the density $\S$ is only an ingredient of the physically relevant definition. It is {\sl a priori} not clear what definition for the 
physical entropy one should assume, as there is not yet a statistical mechanical framework for gravitational dof, and, on the other hand, an information theory  inspired formula or even a thermodynamically inspired formula are difficult to be implemented. Clifton's et al. attempt \cite{ell} belongs to the second framework, and requires to define a temperature field associated with every manifold which is taken into account. Even if this is a viable suggestion, 
it necessarily requires a (generally non-equilibrium) thermodynamics framework which is an ansatz again. Our choice is to ground our entropy candidate to 
the aforementioned density $\S$, and a first step is represented by an averaging procedure over a space region $U$, which is common to other 
definitions \cite{marozzi} and produces what we could call $S^{av}_U\in[0,1]$. Furthermore, 
we also postulate that our gravitational entropy is associated with the maximal entropy which can be associated with a region of space $U$. Such an entropy should be the black hole entropy. This idea is in turn related with the Bekenstein bound \cite{bek-bound} (see also the review \cite{bousso-rev}). The naive ratio beyond this choice is the following: we can obtain what we could call `normalized entropy' $S_n$, with 
$S_n\in [0,1]$ as the ratio between $S$ and $\max (S)$ (cf. also \cite{rodewald}). The same could be possible with an averaged entropy. Our further postulate is to assume that 
in the gravitational case one may choose $\max (S)$ of a spatial region $U$ bounded by a surface $\Sigma$ as the area $A(\Sigma)$ of that surface (see Section \ref{s-weyent}). To be precise, for a  compact space domain $U\in M^3$ (with two dimensional boundary $\Sigma:=\partial U$), we define the {\em Weyl entropy in $U$} as the averaged integral
$$
\S_{U}:=\frac{\text{Area}(\Sigma)}{\operatorname{Vol}_{g}(U)}\int_{U}\S=\frac{\text{Area}(\Sigma)}{\operatorname{Vol}_{g}(U)}\int_{U}\frac{|\W_{\a\b\g\d}|_{\gbr}}{|\R_{\a\b\g\d}|_{\gbr}}\sqrt{g},
$$
where $\operatorname{Vol}_{g}(U)=\int_{U}\sqrt{g}$ and $\text{Area}(\Sigma)$ is the area of the boundary. Analogously, we define $\S^{\mathbf{pf}}_U$. As a consequence of Theorems \ref{t-mone} e \ref{t-monm} we can show the following (see also Section \ref{s-max}):

\begin{theorem} On every $k$-perfect fluid electric $\alpha$-expanding region $\Pa$ satisfying \eqref{eq-asspar} and \eqref{eq-assU} the Weyl entropy in $U$ $\S^{\bf{pf}}_U$ is monotonically increasing in time. Moreover, maximal Weyl entropy in $U$ at time $t_0$ can only occur in a {\em static vacuum} spacetime region $[t_0,T)\times U$.
\end{theorem}

\begin{theorem} On every $k$-perfect fluid magnetic $\alpha$-expanding region $\Pa$ satisfying \eqref{eq-assU} the Weyl entropy in $U$ $\S_U$ is monotonically increasing in time. Moreover, maximal Weyl entropy in $U$ at time $t_0$ can only occur in a {\em flat (vacuum)} spacetime region $[t_0,T)\times U$.
\end{theorem}

As to the list of properties stated in \cite{ell} and summarized above, (E1) is trivially satisfied by both $\S$ and $\S^{\bf{pf}}$. 

(E2) is automatically satisfied by $\S$
but in the case of $\S^{\bf{pf}}$ one must also require $\mathbf{s_{crit}}$, i.e. $\alpha=\frac{1}{3}$. Moreover, we can show that minimal Weyl entropy in $U$ $\S^{\mathbf{pf}}_U$ can only occur in a FLRW spacetime region $[0,T]\times U$ (see Section \ref{s-max}). We will avoid this unphysical situation assuming that the spacetime is non-homogeneous for $t>0$.

As to (E3), local anisotropy implies the presence of a privileged direction in the orthogonal direction with respect to a congruence of timelike curves with unit tangent vectors. In other terms, a space section is isotropic at $x$ if there is no privileged direction in the tangent space at $x$. This is possible only for constant spatial curvature. As a consequence, any deviation from this condition ensures local anisotropy, which is easily implemented in our picture. 

Property (E4) in line of  principle would seem to be automatically satisfied e.g. for static black holes in vacuum. Actually, this identification would not be correct on the grounds of our hypothesis of expanding region, which, as a black hole is expected to form in a collapse process, is hardly compatible with the forming of a black hole itself. 
Still, it can be possible to justify the entropy of a cosmological horizon (if any), which is again of the Bekenstein-Hawking type, i.e. satisfies the 
well-known area law. In this sense, our claim is that our definition can satisfy the area law at least for cosmological horizons. In other terms, 
our constructive ansatz allows us to match by construction and with a suitable choice of a proportionality constant condition (E4) of \cite{ell}.

As a consequence of the previous results, property (E5) is satisfied in pure spacetime regions (i.e. purely electric/magnetic regions).

\medskip

There are two further considerations which appear to be relevant to the present discussion. In agreement with the maximal character of entropy 
in equilibrium, we find that entropy is maximal for static vacuum solutions, i.e. in the case of metrics which can be considered equilibrium states of the geometry, 
which remain static (no further evolution). 
We think that it could be suitable introducing a further condition (E6): entropy is maximal in the case of static solutions in vacuum. 

\medskip

To conclude, it is worthwhile to observe that all aforementioned properties, apart for (E4), hold true also for the Weyl entropy in $U$ simply defined by
$$
\int_{U}\S=\int_{U}\frac{|\W_{\a\b\g\d}|_{\gbr}}{|\R_{\a\b\g\d}|_{\gbr}}\sqrt{g}.
$$

\

\section{Preliminaries on the Curvature of  Spacetimes}

Let $(\X,\gb)$ be a spacetime. We will denote by $\D$ the covariant derivative with respect to $\gb$ and $\R_{\a\b\g\d}$, $\a,\b,\g,\d=0,1,2,3$, its curvature tensor. We recall the following decomposition of the curvature tensor (see \cite{bes, catmas})
$$
\R_{\a\b\g\d} = \W_{\a\b\g\d}+\frac12\left(\R_{\a\g}\gb_{\b\d}-\R_{\a\d}\gb_{\b\g}+\R_{\b\d}\gb_{\a\g}-\R_{\b\g}\gb_{\a\d}\right)-\frac{\R}{6}\left(\gb_{\a\g}\gb_{\b\d}-\gb_{\a\d}\gb_{\b\g}\right)
$$
where $\W_{\a\b\g\d}$, $\R_{\a\b}$ and $\R$ denote the Weyl, Ricci and scalar curvature of $\gb$, respectively. The Einstein equation of spacetime states that
$$
\R_{\a\b}-\frac{1}{2} \R\,\gb_{\a\b} = \T_{\a\b}\,.
$$
Tracing the equation we obtain
$$
\R = - \T
$$
where $\T=\gb^{\a\b}\T_{\a\b}$ is the trace of the stress-energy tensor. Thus
\begin{equation}\label{eq-einnv}
\R_{\a\b} = \T_{\a\b}-\frac12 \T\gb_{\a\b}\,.
\end{equation}
Therefore, one has
\begin{equation}\label{eq-rm}
\R_{\a\b\g\d} = \W_{\a\b\g\d}+\frac12\left(\T_{\a\g}\gb_{\b\d}-\T_{\a\d}\gb_{\b\g}+\T_{\b\d}\gb_{\a\g}-\T_{\b\g}\gb_{\a\d}\right)-\frac{\T}{3}\left(\gb_{\a\g}\gb_{\b\d}-\gb_{\a\d}\gb_{\b\g}\right)
\end{equation}
From the Einstein equation we have that the stress-energy tensor is diverge free, i.e.
\begin{equation}\label{eq-div0}
\D_{\a} \T_{\a\b} = 0 \,.
\end{equation}
We will use the usual notation to denote the norm of a $(0,r)$-tensor $\mathbf{P}$ with respect to the spacetime metric $\gb$, namely
$$
|\mathbf{P}|^2 := \gb^{i_{1}m_{1}}\cdots \gb^{i_{r}m_{r}} \mathbf{P}
_{i_{1}\dots i_{r}} \mathbf{P}_{m_{1}\dots m_{r}}\,.
$$
Note that one has
$$
|\R_{\a\b}|^{2}=|\T_{\a\b}|^{2}\,,
$$
%$$
%\Rb_{\a\b}:=\R_{\a\b}-\frac{\R}{4}\gb_{\a\b}=\left(\L-\frac12 \T\right)\gb_{\a\b}+\T_{\a\b}-\frac{4\L - \T}{4}\gb_{\a\b}=\T_{\a\b}-\frac{\T}{4}\gb_{\a\b}\,,
%$$
%$$
%|\Rb_{\a\b}|^{2}=|\T_{\a\b}|^{2}-\frac{1}{4}\T^2\,.
%$$
and
\begin{align}\label{rm2}
|\R_{\a\b\g\d}|^{2}&=|\W_{\a\b\g\d}|^{2}+2|\R_{\a\b}|^{2}-\frac13 \R^{2} = |\W_{\a\b\g\d}|^{2} + 2|\T_{\a\b}|^{2}-\frac13\T^{2}\\\nonumber
&= |\W_{\a\b\g\d}|^{2} + |\A_{\a\b\g\d}|^2\,,
\end{align}
where 
\begin{equation}\label{eq-A}
|\A_{\a\b\g\d}|^2 = 2|\T_{\a\b}|^{2}-\frac13\T^{2}
\end{equation}
and the four tensor $\A$ is computed from the so called Schouten tensor (see, for instance, \cite[Chapter 1]{catmas})
$$
\A_{\a\b}:=\R_{\a\b}-\frac{\R}{6}\gb_{\a\b}
$$
and is given by
$$
\A_{\a\b\g\d} = \frac12\left(\A_{\a\g}\gb_{\b\d}-\A_{\a\d}\gb_{\b\g}+\A_{\b\d}\gb_{\a\g}-\A_{\b\g}\gb_{\a\d}\right)\,.
$$
The Cotton tensor is given by
\begin{align}\label{eq-cot}\nonumber
\C_{\a\b\g} &:= \D_{\g}\A_{\a\b} - \D_{\b}\A_{\a\g}\\
&=\D_{\g}\R_{\a\b} - \D_{\b}\R_{\a\g}  -
\frac{1}{6}  \left( \D_{\g}\R  \gb_{\a\b} -  \D_{\b}\R
\gb_{\a\g} \right) \\
&= \D_{\g}\T_{\a\b} - \D_{\b}\T_{\a\g}-
\frac{1}{3}  \left( \D_{\g}\T  \gb_{\a\b} -  \D_{\b}\T
\gb_{\a\g} \right)
\end{align}
We recall also the following useful formula (a second bianchi identity for the Weyl tensor, see \cite[Chapter 2]{catmas})
\begin{align}\label{eq-sbinv}
\D_{\e} \W_{\a\b\g\d}+\D_{\g} \W_{\a\b\d\e}+\D_{\d} \W_{\a\b\e\g} &= \frac12 \left(\C_{\a\d\e}\gb_{\b\g}+\C_{\a\e\g}\gb_{\b\d}+\C_{\a\g\d}\gb_{\b\e}\right) \\ \nonumber
&\quad-\frac12 \left(\C_{\b\d\e}\gb_{\a\g}+\C_{\b\e\g}\gb_{\a\d}+\C_{\b\g\d}\gb_{\a\e}\right)
\end{align}
which will be crucial to compute the evolution in time of the Weyl curvature.

\

\section{Globally Hyperbolic Spacetimes}\label{s-glo}

Let $(\X,\gb)$ be a spacetime. It is well known that the causal structure of an arbitrary spacetime can have undesirable pathologies. All these can be avoided by postulating the existence of a Cauchy hypersurface $M^3$ in $\X$, i.e. a hypersurface $M^3$ with the property that any causal curve intersects it at
precisely one point. Spacetimes with this property are called {\em globally hyperbolic} and we will always assume that the space time is the {\em maximal} smooth Cauchy development of initial data on the Cauchy hypersurface $M^3$ (see \cite{bruger} and \cite[Chapter 7]{hawell}). Such spacetimes are in particular stable causal, i.e. they allow the existence of a globally defined differentiable function $t$ whose gradient $\D t$ is everywhere time-like. To be more general, we will consider also the case where the gradient $\D t$ could be light-like on some subsets of $\X$. In this case we will say that the spacetime is {\em almost globally hyperbolic}. This allows the presence of possible horizons. We call $t$ a {\em time function} and the foliation given by its level surfaces a $t$-foliation. We denote by $T$ the future directed unit normal to the foliation.
Topologically, a space-time foliated by the level surfaces of a time function is diffeomorphic to a product manifold $I \times M^3$ where $I=[0,T)$ and $M^3$ is a three-dimensional smooth manifold. In fact the spacetime can be parametrized by points on the slice $t=0$ by following the integral curves of $\D t$. Relative to this parametrization the spacetime metric $\gb$ takes the form
\begin{equation}\label{eq-gam}
\gb=-N^2(t,x)dt^2+g_{ij}(t,x)dx^i dx^j 
\end{equation}
where $t\in I$ and $x=(x^1,x^2,x^3)$ are arbitrary coordinates on the slice $t=0$. The function $N(t,x)=\gb(\D t,\D t)^{-1/2}$ is called the {\em lapse function} of the foliation and $g_{ij}$ its first fundamental form. We will denote by $M_t=\{t\}\times M^3$ the leaves of the foliation. The unit normal to the foliation $T$ is given by $T=N^{-1}\p_t$. The second fundamental form $h$ of the foliation is given by
\begin{equation}\label{eq-sf}
h_{ij}=-\frac{1}{2N}\p_t g_{ij}.
\end{equation}
We denote by $\nabla$ the covariant derivative on the leaves $M_t$ and by $R_{ijkl}$, $R_{ij}$ and $R$ its Riemann, Ricci  and scalar curvature, respectively. Since $M^3$ is three-dimensional, the Riemann tensor of $g$ can be totally recovered from the Ricci tensor, and we have the decomposition (see for instance \cite{catmas})
\begin{equation}\label{eq-rm3}
R_{ijkl} = \left(R_{ik}g_{jl}-R_{il}g_{jk}+R_{jl}g_{ik}-R_{jk}g_{il}\right)-\frac{R}{2}\left(g_{ik}g_{jl}-g_{il}g_{jk}\right).
\end{equation}
By classical formulas, the second fundamental form $h$, the lapse function $N$ and the curvature $R_{ijkl}$ of the foliation are connected to the spacetime curvature tensor $\R_{\a\b\g\d}$ by the following (for instance, see \cite{chrkla})
\begin{align*}
\R_{ijkl} &= R_{ijkl}+ h_{ik}h_{jl}-h_{il}h_{jk} \\
\R_{Tijk} &= \nabla_j h_{ik}-\nabla_k h_{ij} \\
N \R_{TiTj} &= \p_t h_{ij}+ N h_{ik}h_{kj}+\nabla_i \nabla_j N\,,
\end{align*}
where $\R_{ijkl}=\R(\p_i,\p_j,\p_k,\p_l)$, $\R_{Tijk}=\R(T,\p_i,\p_j,\p_k)$ and $\R_{TiTj}=\R(T,\p_i,T,\p_j)$ are the components of the spacetime curvature relative to arbitrary coordinates on $M^3$. Tracing the previous equations we get
\begin{align*}
N \R_{ij} &= N R_{ij} -\p_t h_{ij} +N H h_{ij}-2N h_{il}h_{jl} -\nabla_i \nabla_j N\\
\R_{Tj} &= \nabla_j H-\nabla_k h_{jk} \\
N \R_{TT} &= \p_t H- N |h|^2+\Delta N\,,
\end{align*}
where $H$ denotes the mean curvature of the foliation, namely $H=g^{ij}h_{ij}$. Imposing the Einstein equation \eqref{eq-ein} we obtain 
\begin{align}\nonumber
N\left(\T_{ij}-\frac12\T  g_{ij}\right) &= N R_{ij} -\p_t h_{ij} +N H h_{ij}-2N h_{il}h_{jl} -\nabla_i \nabla_j N\\\label{eq-ricvac}
\T_{Tj} &= \nabla_j H-\nabla_k h_{jk} \\\nonumber
N\left(\T_{TT}+\frac12\T\right)  &= \p_t H- N |h|^2+\Delta N\,.
\end{align}
In particular, tracing these equations, we discover the so called {\em Einstein constrained equations} for the foliation of a  spacetime
\begin{align}\label{eq-coneq}
R + H^2 -|h|^2 &= 2\T_{TT} \\
\nabla_j H-\nabla_k h_{jk}&=\T_{Tj}\,.
\end{align}
Using \eqref{eq-rm}, since 
$$
\gb_{ij}=g_{ij},\quad \gb_{Ti}=0\quad \text{and}\quad \gb_{TT}=-1\,, 
$$
we have
\begin{align*}
\R_{ijkl} &= \W_{ijkl}+\frac12\left(\T_{ik}g_{jl}-\T_{il}g_{jk}+\T_{jl}g_{ik}-\T_{jk}g_{il}\right)-\frac{\T}{3}\left(g_{ik}g_{jl}-g_{il}g_{jk}\right) \\
\R_{Tijk} &= \W_{Tijk}+\frac12\left(\T_{Tj}g_{ik}-\T_{Tk}g_{ij}\right)\\
\R_{TiTj} &= \W_{TiTj}+\frac12\left(\T_{TT}g_{ij}-\T_{ij}\right)+\frac{\T}{3}g_{ij}
\end{align*}
and, from \eqref{eq-rm3}, we obtain
\begin{align}\label{eq-g1}
\W_{ijkl} &= -\frac12\left(\T_{ik}g_{jl}-\T_{il}g_{jk}+\T_{jl}g_{ik}-\T_{jk}g_{il}\right)\\\nonumber
&\quad +\left(R_{ik}g_{jl}-R_{il}g_{jk}+R_{jl}g_{ik}-R_{jk}g_{il}\right)\\\nonumber&\quad-\frac{3R-2\T}{6}\left(g_{ik}g_{jl}-g_{il}g_{jk}\right)+ h_{ik}h_{jl}-h_{il}h_{jk} \\\label{eq-g2}
\W_{Tijk} &= -\frac12\left(\T_{Tj}g_{ik}-\T_{Tk}g_{ij}\right)+\nabla_j h_{ik}-\nabla_k h_{ij} \\\label{eq-g3}
N \W_{TiTj} &= -\frac12N\left(\T_{TT}g_{ij}-\T_{ij}\right)-\frac{\T}{3}Ng_{ij}+\p_t h_{ij}+ N h_{ik}h_{kj}+\nabla_i \nabla_j N\,,
\end{align}

Given a globally hyperbolic spacetime  $(\X,\gb)=(I\times M^3, -N^2 dt^2+g)$, being the vector field $T$ a timelike unit vector, following the $1+3$ covariant description of gravitational fields \cite{ehl} and in analogy to the decomposition of the Maxwell tensor into electric and magnetic parts, the Weyl tensor $\W$ can also be decomposed into electric and magnetic parts as
$$
\E_{\alpha\beta}=\W_{\alpha\gamma\beta\delta}T^{\gamma}T^{\delta},\quad \H_{\alpha\beta}=\frac12\bm{\varepsilon}_{\alpha\gamma\delta\eta}\W^{\delta\eta}_{\quad\beta\theta}T^{\gamma}T^{\theta}
$$   
where $\bm{\varepsilon}_{\alpha\gamma\delta\eta}$ is the volume element. Using arbitrary coordinates on $M^3$, we get
$$
\E_{ij}=\W_{TiTj},\,\E_{\alpha T}=0, \quad \H_{ij}=\frac12\bm{\varepsilon}_{iT\delta\eta}\W^{\delta\eta}_{\quad jT}=\frac12\bm{\varepsilon}_{iTkl}\W^{kl}_{\quad jT},\,\H_{\alpha T}=0.
$$
In particular, on a globally hyperbolic spacetime, $\E$ vanishes if and only if $\W_{TiTj}=0$ locally, while $\H$ vanishes if and only if $\W_{Tijk}=0$ locally. Spacetimes with zero electric part $\E$ (or magnetic part $\H$) are called in the literature {\em Pure Electric (or Magnetic) spacetimes} (see, for instance \cite{mac, bon2, her}).

We note that, one has
$$
|\W_{\alpha\beta\gamma\delta}|^2=\frac{1}{4}\left(|\E_{\alpha\beta}|^2-|\H_{\alpha\beta}|^2\right),\quad |\W_{\alpha\beta\gamma\delta}|_{\gbr}^2=\frac{1}{4}\left(|\E_{\alpha\beta}|^2+|\H_{\alpha\beta}|^2\right)\,.
$$
In particular, the quantity $W$ considered in \cite{ell} constructed from the so called Bel-Robinson tensor \cite{bel}, coincides exactly with $|\W_{\alpha\beta\gamma\delta}|_{\gbr}^2$. 

It may be noted that the definition of electric/magnetic part of the Weyl tensor depends on $T$, and then is observer dependent, as $T$ may be 
interpreted as the four-velocity of some observer, to be defined as ``Eulerian observer'', also called fiducial observer. Still,  albeit different timelike congruences experience different decompositions, the tensorial nature of the equations defining pure spacetime regions ensures covariance.

\

\section{A Weyl Entropy}\label{s-weyent}

Throughout this section we will consider a globally hyperbolic spacetime 
$$
(\X,\gb)=(I\times M^3, -N^2 dt^2+g)
$$ 
as discussed in the previous section. On a generic spacetime the functions $|\W_{\a\b\g\d}|^2$ can be negative somewhere. In fact, while the norm of a symmetric two tensor (such as the Ricci tensor) is always nonnegative, a simple computation (see Lemma \ref{l-for}) gives 
\begin{equation}\label{eq-nw}
|\W_{\a\b\g\d}|^2=-4|\W_{Tijk}|^2+4|\W_{TiTj}|^2+|\W_{ijkl}|^2=-4|\W_{Tijk}|^2+8|\W_{TiTj}|^2\,.
\end{equation}
We ask for the   entropy to be measured from the Weyl tensor, to be nonnegative and to be zero if and only if the Weyl tensor vanishes. Thus we will consider the norm computed with respect to the Riemannian metric $\gbr$ associated to $\gb$, namely:
$$
\gbr := N^{2} dt^{2}+g \,.
$$
Thus, we have
$$
|\W_{\a\b\g\d}|_{\gbr}^2=4|\W_{Tijk}|^2+8|\W_{TiTj}|^2 
$$
which is nonnegative and vanishes if and only if $\W=0$. We give the following definition of Weyl entropy:

\begin{definition} At given point $(t,x)\in \X$ we define the {\em Weyl entropy} $\S=\S(t,x)$ as
\begin{align*}
\S &:= \frac{|\W_{\a\b\g\d}|_{\gbr}}{|\R_{\a\b\g\d}|_{\gbr}}\sqrt{g}\,,
\end{align*}
where $\R_{\a\b\g\d}\neq 0$ and $\S=\sqrt{g}$ where $\R_{\a\b\g\d}=0$. Here $\sqrt{g}$ denotes the square-root of the determinant of the space metric $g$. We will call 
$$
{\bf s}:=
\begin{cases}
\frac{|\W_{\a\b\g\d}|_{\gbr}}{|\R_{\a\b\g\d}|_{\gbr}}&\quad\text{if}\quad\R\neq 0 \\
1 &\quad\text{if}\quad\R=0 
\end{cases}
$$
the {\em  Weyl entropy density}.
\end{definition}
Given a compact space domain $U\in M^3$ (with two dimensional smooth boundary $\Sigma:=\partial U$), we define the {\em Weyl entropy in $U$} as the averaged integral
$$
\S_{U}:=\frac{\text{Area}(\Sigma)}{\operatorname{Vol}_{g}(U)}\int_{U}\S=\frac{\text{Area}(\Sigma)}{\operatorname{Vol}_{g}(U)}\int_{U}{\bf s}\sqrt{g},
$$
where $\operatorname{Vol}_{g}(U)=\int_{U}\sqrt{g}$ and $\text{Area}(\Sigma)$ is the two-dimensional Haussdorf  area of the boundary, computed with respect to the metric induced by $g$ on $\Sigma$. By definition, $\S_U\geq 0$, with equality if and only if the Weyl tensor vanishes on $U$. Moreover, from Lemma \ref{l-for}, we have
$$
\S_{U} \leq \text{Area}(\Sigma)
$$
with equality if and only if the Ricci tensor of $\gb$ vanishes. From the Einstein equation, this is equivalent to say that the stress-energy tensor $\T$ vanishes (vacuum region). In particular we have the following characterization of space region with maximal   Weyl entropy: the Weyl entropy in $U$, $\S_U$ is maximal at a given time $t$ if and only if $
\T = 0$ on $\{t\}\times U$.

\medskip

We stress again that, in the above definition, the appearance of the factor $\text{Area}(\Sigma)$ is an ansatz which is seemingly arbitrary, but 
it can be justified on the grounds of black hole thermodynamics, and is compatible with the conjecture that the maximal entropy for a 
region with boundary area $\text{Area}(\Sigma)$ is the black hole entropy, which is $A/4$, as well known. In this sense, in the definition we could also 
introduce a multiplicative constant $\zeta$ (with $\zeta=1/4$ for static black holes), in such a way to implement fully the assumption (E4) of \cite{ell}.

\medskip

As well known, in thermodynamics the entropy is maximal for equilibrium states. Of course, it is not possible to associate the Weyl entropy with a thermodynamic entropy, even in a non-equilibrium framework, as the possibility to define a local temperature in a natural way is missing (cf. anyway the picture in \cite{ell}). Still, some `entropy-like' properties 
can be identified. In fact, in Section \ref{s-max}, we show that the Weyl entropy related to region where there is monotonicity is maximal at some time, if and only if the region is a static vacuum solution, as if it were a real equilibrium entropy (an equilibrium state is a static one, and static solutions appear as 
sort of 'equilibrium states of the geometry').

\medskip

In order to simplify the computations for the evolution of the Weyl entropy, we introduce the following quantity, where $\A\neq 0$, 
\begin{equation}\label{ge-wge}
{\bf s^{2}}=\frac{|\W_{\a\b\g\d}|^{2}_{\gbr}}{|\R_{\a\b\g\d}|^{2}_{\gbr}} = \frac{\bf \bar{s}^{2}}{\bf \bar{s}^{2}+1}\,,
\end{equation}
where
$$
{\bf \bar{s}^{2}}:=\frac{|\W_{\a\b\g\d}|^{2}_{\gbr}}{|\A_{\a\b\g\d}|^{2}_{\gbr}}\,.
$$

For reasons that will be clear in the rest of the computations, we will study the entropy $\S$ on spacetimes $(\X,\gb)$ satisfying (locally)
either
\begin{equation}\label{H1}\tag{H1}
\W_{Tijk}\equiv 0
\end{equation}
or
\begin{equation}\label{H2}\tag{H2}
\W_{TiTj}\equiv 0\,.
\end{equation}
As observed in the previous section, condition \eqref{H1} is equivalent for the spacetime to be {\em Pure Electric} (also called {\em Coulomb like}), while condition \eqref{H2} is equivalent for the spacetime to be {\em Pure Magnetic} (or {\em Wave Like}). From \eqref{eq-nw} it is clear that
\begin{equation}\label{eq-sss}
|\W_{\a\b\g\d}|_{\gbr}=\begin{cases} |\W_{\a\b\g\d}|&\quad\text{if }\eqref{H1}\text{ holds}\,, \\ -|\W_{\a\b\g\d}|&\quad\text{if }\eqref{H2}\text{ holds}\,.\end{cases}
\end{equation}

\

Geometrically, from equation \eqref{eq-g2} there is an equivalence between condition \eqref{H1} and the second fundamental form $h_{ij}$ to be a {\em Codazzi tensor} on the space slice $M^3$, whenever the stress-energy tensor is diagonal (e.g. perfect fluid case). In this case, sufficient conditions that imply \eqref{H1} are $h_{ij}\equiv 0$ (totally geodesic foliation, i.e. time-symmetric spacetime) or, more in general, $h_{ij}\equiv\frac{H}{3}g_{ij}$ (totally umbilical foliation). A sufficient condition is also $\nabla h\equiv 0$, i.e. parallel second fundamental form. Many examples satisfying this assumption are well studied (for instance Bianchi type I, Lemaitre-Tolman \cite{ell, gro}). Examples of spacetime satisfying \eqref{H2} can be found in \cite{ell2, ari}.

\

\section{Pure Electric Spacetimes: Perfect Fluids}

\noindent Take a {\em perfect fluid} stress-energy tensor given by
\begin{equation}\label{eq-pf}
\T_{TT}=M, \quad \T_{Ti}=0, \quad \T_{ij}=P g_{ij}\,,
\end{equation}
where $M=M(t,x)$ is the mass density and $P=P(t,x)$ the pressure of the fluid. One has
$$
\T :=\gb^{\a\b}\T_{\a\b}= 3P - M\,
$$
and
$$
|\T_{\a\b}|^2=M^2+3P^2\,.
$$
Note that, from \eqref{eq-div0} one has
$$
0 = -\D_T \T_{TT} + \D_i \T_{iT} = -\D_T M -\Gb_{ii}^T \T_{TT} -\Gb_{iT}^j \T_{ij}= -\D_T M + H M + H P \,,
$$
i.e.
\begin{equation*}
\D_T M = (M+P) H \,.
\end{equation*}
Recall that $\D_T$ acts as $N^{-1}\p_t$ on functions. Moreover, from \eqref{eq-div0} one has 
$$
0 = -\D_T \T_{Tj} + \D_i \T_{ij} = \Gb_{TT}^i \T_{ij} + \Gb_{Tj}^T \T_{TT} + \p_j P = \p_j P\,,
$$
i.e.
\begin{equation}\label{eq-P}
\p_j P = 0 \,,
\end{equation}
thus $P$ is a function only of time, $P=P(t)$. In general, it is not a restriction to impose that the pressure and the mass satisfy the following {\em equation of state}:
$$
P = (k-1) M,
$$
for some function $k=k(t,x)$. Particular cases of interest are give by constant values of the type:
\begin{itemize}
\item {\em Radiation Dominated:} $k=\frac43$;
\item {\em Matter Dominated:} $k=1$;
\item {\em Vacuum Energy Dominated:} $k= 0$. 
\end{itemize}
These important cases are very-well known in physical literature, as they represent a standard description of the different epochs for the evolution of the Universe (see Figure 1). In the present modelization, it is possible to allow a local space-time dependence for the parameter $k$, in a generalization of the standard picture which allows an interpolation between the constant values occurring in physical models. In particular, we have
\begin{align}\label{eq-evmass}
\D_T M &= k H M \,.
\end{align}
%Note that, as a consequence we have that the mean curvature of the foliation is constant in space, $H=H(t)$. 
We will use the following notation for the evolution of $k$:
$$
\
\D_T k =k' H
$$
where $H\neq 0$ and we set $k'=0$ whenever $H=0$. From \eqref{rm2} we get

\begin{align}\label{eq-rm2pf}\nonumber
|\R_{\a\b\g\d}|^{2} &=|\W_{\a\b\g\d}|^{2}+2|\R_{\a\b}|^{2}-\frac13 \R^{2} = |\W_{\a\b\g\d}|^{2} + 2|\T_{\a\b}|^{2}-\frac13\T^{2}\\
&= |\W_{\a\b\g\d}|^{2} + \frac53 M^2+3P^2+2MP\\\nonumber
&=|\W_{\a\b\g\d}|^{2} +\frac{9k^2-12k+8}{3}M^2\\\nonumber
&=|\W_{\a\b\g\d}|^{2} +|A_{\a\b\g\d}|^2\,,
\end{align}
and 
$$
|\A_{\a\b\g\d}|^2 = \frac{9k^2-12k+8}{3}M^2\,.
$$
Note also that $\A_{Tijk}=0$, since $\R_{Ti}=\T_{Ti}=0$. From \eqref{eq-evmass} we have
\begin{equation}\label{eq-evA}
\D_T |\A_{\a\b\g\d}|^2 = \frac{2k(9k^2-12k+8)+6k'(3k-2)}{3} H\,M^2\,.
\end{equation}

On a perfect fluid, since
$$
\T_{TT}=M,\quad \T_{iT}=0, \quad \T_{ij}=P g_{ij}\quad\text{and}\quad \T=3P -M\,,
$$ 
we have
\begin{align*}
\C_{ijT} &= \D_{T}\T_{ij} - \D_{j}\T_{iT}-
\frac{1}{3}  \left( \D_{T}\T  \gb_{ij} -  \D_{j}\T
\gb_{iT} \right) \\
& = (\D_T P)g_{ij}+\Gb_{ij}^T \T_{TT}+\Gb_{jT}^p \T_{ip} - \frac13 \left(3\D_T P - \D_T M\right)g_{ij} \\
&= -(M+P) h_{ij} + \frac13 (\D_T M) g_{ij}\,.
\end{align*}
From Proposition \ref{p-firnv} we obtain
\begin{equation}\label{eq-grapfgen}
\frac12 \D_T |\W_{\a\b\g\d}|^2 = 16 H |\W_{TiTj}|^2-24h_{jl}\W_{TiTj}\W_{TiTl}+4(M+P)h_{ij}\W_{TiTj}\,,
\end{equation}
and, if $P = (k-1) M$, then
\begin{equation}\label{eq-grapf}
\frac12 \D_T |\W_{\a\b\g\d}|^2 = 16 H |\W_{TiTj}|^2-24h_{jl}\W_{TiTj}\W_{TiTl}+4kMh_{ij}\W_{TiTj}\,,
\end{equation}

As a corollary, we have the following formula for the evolution of the  Weyl entropy of  Pure Electric Perfect Fluids:

\begin{proposition}\label{p-evopf} Let $(\X,\gb)$ be a globally hyperbolic perfect fluid satisfying \eqref{H1}. Then, the  Weyl entropy $\S$ satisfies
\begin{align}\label{eq-evoSvark}
\D_T \S &= \frac{|\W| |\A|^2}{|\R|_{\gbr}^3}\Big[-kH-H\frac{|\W|^2}{|\A|^2} - 24 \frac{\mathring{h}_{jl}\W_{TiTj}\W_{TiTl}}{|\W|^2}+4k M \frac{\mathring{h}_{ij}\W_{TiTj}}{|\W|^2} \nonumber \\ &\quad-H\frac{k'(3k-2)M^2}{|\A|^2}\Big]\sqrt{g}\,,
\end{align}
\end{proposition}
\begin{proof} First of all, we note that if at some point $\A=0$ (i.e. $M=0$), then $\mathbf{s}=1$. Since $\mathbf{s}\leq1$, $\mathbf{s}$ attains a maximum and $\D_{T}\mathbf{s}=0$. In this case (see \eqref{eq-evdet} below) 
\begin{equation}\label{eq-coin}
\D_{T}\S =- H \sqrt{g}.
\end{equation}
On the other hand, where $\A \neq 0$ using \eqref{ge-wge} and \eqref{eq-nw}, we obtain
$$
\D_T \S = \D_{T}\left(\mathbf{s}\sqrt{g}\right)=\left(\frac{\mathbf{s}}{\bf{\bar{s}}}\right)^3\D_T \mathbf{\bar{s}}\,\sqrt{g} + \mathbf{s}\,\D_T \sqrt{g} \,.
$$
Since, from Lemma \ref{l-for} $|\A|_{\gbr}^2=|\A|^2$ and, from \eqref{eq-rm2pf} and \eqref{eq-evA}, we have
$$
|\A|^2=\frac{9k^2-12k+8}{3}M^2
$$
and
$$
\D_T |\A| = \frac{k(9k^2-12k+8)+3k'(3k-2)}{3|\A|} H\,M^2\,.
$$
Thus, from \eqref{eq-sss},  \eqref{eq-grapf} and Lemma \ref{l-for} we obtain
\begin{align*}
\D_T \bf{\bar{s}}&=\D_T \left(\frac{|\W|}{|\A|}\right) = \frac{1}{|\A|^2}\left(|\A| \D_T |\W| - |\W| \D_T |\A|\right) \\
&=\frac{1}{|\A|^2} \Big[ \frac{|\A|}{|\W|}\left(H |\W|^2-24 \mathring{h}_{jl}\W_{TiTj}\W_{TiTl}+4kM \mathring{h}_{ij}\W_{TiTj} \right)\\ &\quad-H\frac{|\W|}{3|\A|}\left[k(9k^2-12k+8)+3k'(3k-2)\right]M^2\Big] \\
&= H \frac{|\W|}{|\A|} - 24 \frac{\mathring{h}_{jl}\W_{TiTj}\W_{TiTl}}{|\W||\A|}+4k M \frac{\mathring{h}_{ij}\W_{TiTj}}{|\W||\A|}\\ &\quad-H\frac{|\W|}{3|\A|^3}\left[k(9k^2-12k+8)+3k'(3k-2)\right]M^2 \,.
\end{align*}
From \eqref{eq-sf} and the well know formula for the variation of the determinant, one has
\begin{align}\label{eq-evdet}
\D_T \sqrt{g} = \frac{1}{2N}\left(g^{ij}\p_t g_{ij}\right)\sqrt{g} = - H \sqrt{g}\,.
\end{align}
Using Lemma \ref{l-for}, we get
\begin{align*}
\D_T \S &= \frac{|\A|^3}{|\R|_{\gbr}^3}\Big[ H \frac{|\W|}{|\A|} - 24 \frac{\mathring{h}_{jl}\W_{TiTj}\W_{TiTl}}{|\W||\A|}+4k M \frac{\mathring{h}_{ij}\W_{TiTj}}{|\W||\A|}\\ &\quad-H\frac{|\W|}{3|\A|^3}\left[k(9k^2-12k+8)+3k'(3k-2)\right]M^2-H \frac{|\W||\R|_{\gbr}^2}{|\A|^3}\Big]\sqrt{g} \\
&= \frac{|\W| |\A|^2}{|\R|_{\gbr}^3}\Big[-H\left(\frac{|\R|_{\gbr}^2}{|\A|^2}-1\right) - 24 \frac{\mathring{h}_{jl}\W_{TiTj}\W_{TiTl}}{|\W|^2}+4k M \frac{\mathring{h}_{ij}\W_{TiTj}}{|\W|^2}\\ &\quad-H\frac{1}{3|\A|^2}\left[k(9k^2-12k+8)+3k'(3k-2)\right]M^2\Big]\sqrt{g} \\
&=\frac{|\W| |\A|^2}{|\R|_{\gbr}^3}\Big[-H\frac{|\W|^2}{|\A|^2} - 24 \frac{\mathring{h}_{jl}\W_{TiTj}\W_{TiTl}}{|\W|^2}+4k M \frac{\mathring{h}_{ij}\W_{TiTj}}{|\W|^2}\\ &\quad-H\frac{1}{3|\A|^2}\left[k(9k^2-12k+8)+3k'(3k-2)\right]M^2\Big]\sqrt{g}\,.
\end{align*}
Since
$$
\frac13\left[k(9k^2-12k+8)+3k'(3k-2)\right]M^2=k|\A|^2+k'(3k-2)M^2\,,
$$
we obtain
\begin{align*}
\D_T \S &= \frac{|\W| |\A|^2}{|\R|_{\gbr}^3}\Big[-kH-H\frac{|\W|^2}{|\A|^2} - 24 \frac{\mathring{h}_{jl}\W_{TiTj}\W_{TiTl}}{|\W|^2}+4k M \frac{\mathring{h}_{ij}\W_{TiTj}}{|\W|^2}\\ &\quad-H\frac{k'(3k-2)M^2}{|\A|^2}\Big]\sqrt{g}\,,
\end{align*}
and this concludes the proof. Note that this formula coincides with \eqref{eq-coin} when $\A=0$.
\end{proof}

\

\section{Monotonicity of the Weyl Entropy in Perfect Fluid Electric Regions}\label{s-mon}

In this section we will assume the spacetime $(\X,\gb)$ to be almost globally hyperbolic. In fact, we note that all the computations done in the previous sections are local, so they hold on every open subset of $\X$ where the lapse function $N$ is strictly positive and the Weyl tensor $\W$ satisfies \eqref{H1}.

\begin{definition}\label{d-pfe}  Let $(\X,\gb)=(I\times M^3, -N^2 dt^2+g)$ be an almost globally hyperbolic spacetime, let $U\subset M^3$ and $I'\subset I$ be two open sets. We say that $\P:=I'\times U$ is a {\em $k$-perfect fluid electric region} if 
\begin{itemize}

\item[(1)] the stress-energy tensor $\T$ is given by \eqref{eq-pf} on $\P$, with 
$$
P=(k-1)M,\quad k=k(t,x)\in\left[0,\frac43\right];
$$ 

\item[(2)] the Weyl tensor $\W$ satisfies \eqref{H1} on $\P$;

\item[(3)] the lapse function $N$ is strictly positive on $\P$.

\end{itemize}
\end{definition}

We need also the following definition concerning the expansion and the intrinsic curvature of the space metric:

\begin{definition}\label{d-pfe}  Let $(\X,\gb)=(I\times M^3, -N^2 dt^2+g)$ be an almost globally hyperbolic spacetime, let $U\subset M^3$ and $I'\subset I$ be two empty open sets. We say that $I'\times U$ is a {\em $\alpha$-expanding region} if there exists a function $\alpha=\alpha(t,x)\in\left[0,\frac13\right]$ such that
$$
h_{ij} \leq \alpha H g_{ij} \leq 0 \quad \text{on}\quad I'\times U\,.
$$
In particular, we will denote by $\Pa$ a $k$-perfect fluid electric $\alpha$-expanding region.
\end{definition}

From \eqref{eq-sf} the condition of $\alpha$-expansion is equivalent to
$$
\p_t g_{ij} \geq - \alpha N H g_{ij}\,.
$$
In particular $0$-expansion is equivalent to say that spatial metric is increasing in time, so the (space) region is {\em expanding}. From algebraic reasons, the case $\alpha=\frac13$ implies that the foliation is totally umbilical, i.e. $h_{ij}=\frac{H}{3}g_{ij}$, or equivalently, $\mathring{h}=0$. Note that this holds also if $H$ is zero (minimal foliation), since the $\alpha$-expansion implies $h_{ij}\leq 0$, and thus $h=0$. 

We will assume that the fluid parameters $k,\alpha$, satisfy the following evolution inequalities
\begin{equation}\label{eq-asspar}
0\leq k' \leq 
\frac{k(9k^2-12k+8)}{3(4-3k)}\min\left\{\frac{9\alpha'}{4(1-3\alpha)}, 1\right\},\quad \alpha'\geq 0,
\end{equation}
where $\D_T u=:u' H$ where $H\neq 0$ and $u':=0$ where $H=0$. Note that, since $H\leq 0$, the assumptions \eqref{eq-asspar} imply that $\D_T k\leq 0$ and $\D_T\alpha\leq 0$, which are very natural (see Figure 1).

As a consequence of Proposition \ref{p-evopf} we can show an estimate for the modified  Weyl entropy $\S^{\bf{pf}}$ on a $\Pa$, defined as
$$
\S^{\bf{pf}} := \S+ \mathbf{s_{crit}} \sqrt{g}=\left(\frac{|\W_{\a\b\g\d}|_{\gbr}}{|\R_{\a\b\g\d}|_{\gbr}}+\mathbf{s_{crit}}\right) \sqrt{g}\,,
$$
where
$$
\mathbf{s_{crit}}:=\sqrt{1-3\alpha}\left(\frac{\sqrt{2}}{4}+2k\sqrt{\frac{1-3\alpha}{9k^2-12k+8}}\right)\,.
$$
Note that $\mathbf{s_{crit}}=\mathbf{s_{crit}}(t,x)\geq 0$ and 
$$
\mathbf{s_{crit}}=0 \quad\Longleftrightarrow\quad \alpha=\frac13 \,.
$$
Moreover, we will show in the proof of Theorem \ref{t-mon} that the assumption \eqref{eq-asspar} implies that $\D_T \mathbf{s_{crit}}\geq 0$ and thus 
$$
\D_T \left(\mathbf{s_{crit}} \sqrt{g}\right)\geq 0,
$$ 
which is very natural. The definition of $ \S^{\bf{pf}}$ is such that monotonicity in time is preserved only at the cost to introduce also the 
new contribution $\mathbf{s_{crit}}$, which is related in part to the geometry itself (through its dependence on the parameter $\alpha$) and in part 
on the matter field equation of state (through the parameter $k$). This further term vanishes only when $\alpha=\frac{1}{3}$, i.e. in the 
homogeneous case. It is worthwhile mentioning that, in the standard discussions in literature, $k$ is a constant, and also $\alpha$ 
can be assumed to be constant. As a consequence, in such a situation, the definition given above amounts to shifting the 
ratio $\frac{|\W_{\a\b\g\d}|_{\gbr}}{|\R_{\a\b\g\d}|_{\gbr}}$ by a constant and, moreover, the assumptions \eqref{eq-asspar} on the parameters are automatically satisfied.

\begin{theorem}\label{t-mon} On every $k$-perfect fluid electric $\alpha$-expanding region $\Pa$ satisfying \eqref{eq-asspar} the  entropy $\S^{\bf{pf}}$ is monotonically increasing, i.e.
$$
\D_T \S^{\bf{pf}} \geq 0.
$$
Moreover, the equality holds at some point if and only  if $\D_T \mathbf{s_{crit}}=0$ and either $h=0$, or $|\W|=0$, $\alpha=\frac13$ and $\mathbf{s_{crit}}=0$.
\end{theorem}
\begin{proof} From the $\alpha$-expanding assumption we have 
$$
-24\mathring{h}_{jl}\W_{TiTj}\W_{TiTl}\geq 8(1-3\alpha) H |\W_{TiTj}|^2 = (1-3\alpha) H |W|^2\,,
$$
since $|\W|^2=|\W_{\a\b\g\d}|^2=8|\W_{TiTj}|^2$. Moreover, since $B:=h-\alpha H g \leq 0$, one has $|B|^2 \leq |\text{tr}B|^2$, i.e.
$$
|h-\alpha H g|^2 \leq (1-3\alpha)^2 H^2\,,
$$
or equivalently
$$
|h|^2 \leq (6\alpha^2-4\alpha+1)H^2 \,.
$$
Recalling that $H\leq 0$ and $\alpha\leq\frac13$ we obtain the following estimate
$$
|\mathring{h}| \leq -\sqrt{\frac23}(1-3\alpha) H \,.
$$
In particular, we get
$$
|\mathring{h}_{ij}\W_{TiTj}| \leq - \sqrt{\frac23}(1-3\alpha) H |\W_{TiTj}| = - \frac{1-3\alpha}{2\sqrt{3}} H |\W|
$$
From equation \ref{eq-evoSvark}, we obtain
\begin{align}\label{eq-eqeq3}
\D_T \S &= \frac{|\W| |\A|^2}{|\R|_{\gbr}^3}\Big[-kH-H\frac{|\W|^2}{|\A|^2} - 24 \frac{\mathring{h}_{jl}\W_{TiTj}\W_{TiTl}}{|\W|^2}+4k M \frac{\mathring{h}_{ij}\W_{TiTj}}{|\W|^2}\\ \nonumber&\quad-H\frac{k'(3k-2)M^2}{|\A|^2}\Big]\sqrt{g} \\\nonumber
&\geq \frac{|\W| |\A|^2}{|\R|_{\gbr}^3}\Big[k-1+3\alpha+\frac{|\W|^2}{|\A|^2} -\frac{2k(1-3\alpha)}{\sqrt{3}}\frac{M}{|\W|}\\\nonumber &\quad+\frac{k'(3k-2)M^2}{|\A|^2}\Big](-H)\sqrt{g} \\\nonumber
&= \frac{|\W|}{3|\R|_{\gbr}^3}\Big[(k-1+3\alpha)(9k^2-12k+8)+3k'(3k-2)\Big]M^2(-H)\sqrt{g} \\\nonumber
&\quad+\frac{1}{|\R|_{\gbr}^3}\Big[|\W|^3-\frac{2k(1-3\alpha)}{\sqrt{3}}M|\A|^2 \Big](-H)\sqrt{g} 
\end{align}
Since 
$$
\frac13M^2=\frac{1}{9k^2-12k+8}|\A|^2\,,
$$
we get
\begin{align*}
\D_T \S &= \frac{1}{|\R|_{\gbr}^3}\Big[|\W|^3-(1-3\alpha)|\W||\A|^2-\frac{2k(1-3\alpha)}{\sqrt{9k^2-12k+8}}|\A|^3\Big](-H)\sqrt{g}\\
&\quad+\frac{ M^2 |\W|}{3|\R|_{\gbr}^3}\Big[k(9k^2-12k+8)+3k'(3k-2)\Big](-H)\sqrt{g}.    
\end{align*}
Using the assumption \eqref{eq-asspar} 
$$
0\leq k' \leq \frac{k(9k^2-12k+8)}{3(4-3k)}
$$
it is easy to see that
\begin{equation}\label{eq-eqeq2}
k(9k^2-12k+8)+3k'(3k-2) \geq k(9k^2-12k+8)+3k'(3k-4) \geq 0
\end{equation}
obtaining
\begin{align}\label{eq-eqeq22}
\D_T \S \geq \frac{1}{|\R|_{\gbr}^3}\Big[|\W|^3-(1-3\alpha)|\W||\A|^2-\frac{2k(1-3\alpha)}{\sqrt{9k^2-12k+8}}|\A|^3\Big](-H)\sqrt{g},      
\end{align}
To estimate $|\W||\A|^{2}$, when $\alpha\neq\frac13$, we use Young's inequality
$$
2 a b \leq \frac{1}{\theta^3}a^3 + \theta^{\frac32}b^{\frac32}
$$
which holds for all nonnegative numbers $a,b\geq 0$ and all $\theta>0$. Choosing
$$
\theta^3=\frac{1-3\alpha}{2}
$$
and using that $|\A|^3\leq |\R|_{\gbr}^3-|\W|^3$ (since $|\R|_{\gbr}^2=|\A|^2+|\W|^2$), we get
\begin{align}\label{eq-eqeq}\nonumber
\D_T \S &\geq \frac{1}{|\R|_{\gbr}^3}\left[|\W|^3-(1-3\alpha)|\W||\A|^2-\frac{2k(1-3\alpha)}{\sqrt{9k^2-12k+8}}|\A|^3\right](-H)\sqrt{g} \\\nonumber
&\geq -\sqrt{1-3\alpha}\left(\frac{\sqrt{2}}{4}+2k\sqrt{\frac{1-3\alpha}{9k^2-12k+8}}\right)\frac{|\A|^3}{|\R|_{\gbr}^3}(-H)\sqrt{g} \\
&\geq -\mathbf{s_{crit}}(-H)\sqrt{g} \\\nonumber
&=-\D_T \left(\mathbf{s_{crit}}\sqrt{g}\right) +\D_T\left(\mathbf{s_{crit}}\right)\sqrt{g}.
\end{align}
We claim that the assumption on the fluid parameters \eqref{eq-asspar} implies $\D_T\left(\mathbf{s_{crit}}\right)\geq 0$. In fact, let $A:=\sqrt{1-3\alpha}\geq 0$. Then, since $\alpha'\geq 0$, where $A\neq 0$, we have
$$
A'=-\frac{3\alpha'}{2\sqrt{1-3\alpha}}=-\frac{3\alpha'}{2(1-3\alpha)}A\leq 0\,.
$$
Then
\begin{align*}
\mathbf{s'_{crit}} &=\frac{\sqrt{2}}{4}A'+\frac{4kAA'}{\sqrt{9k^2-12k+8}}+2A^2\left(\frac{2k}{\sqrt{9k^2-12k+8}}\right)'\\
&\leq \frac{4A^2}{\sqrt{9k^2-12k+8}}\left(-\frac{3k\alpha'}{1-3\alpha}+\frac{4k'(4-3k)}{9k^2-12k+8}\right)\leq 0
\end{align*} 
if 
$$
0\leq k' \leq 
\frac{k(9k^2-12k+8)}{3(4-3k)}\left(\frac{9\alpha'}{4(1-3\alpha)}\right).
$$
Thus, from \eqref{eq-asspar} we get $\D_T \mathbf{s_{crit}}\geq 0$. Therefore, we proved that
$$
\D_T \S^{\bf{pf}}=\D_T\S+ \D_T \left(\mathbf{s_{crit}}\sqrt{g}\right) \geq 0.
$$

We verify now the equality case. Clearly, if at some point $\D_T \mathbf{s_{crit}}=0$ and $h=0$, then the quantity $\S^{\bf{pf}}=\S+\mathbf{s_{crit}}\sqrt{g}$ must be constant in time. Moreover, if $\D_T \mathbf{s_{crit}}=0$, $|\W|=0$ and $\alpha=\frac13$ ($\mathbf{s_{crit}}=0$), $\D_{T}\S^{\bf{pf}}=\D_{T}\S=0$. On the other hand, if equality holds in all the estimates of Theorem \ref{t-mon}, then from equation \eqref{eq-eqeq} we get that $\D_T\left(\mathbf{s_{crit}}\right)=0$ and either $\alpha=\frac13$ or $H=0$ or $|\A|^{2}=|\R|_{\gbr}^{2}$, i.e. $|\W|^{2}=0$. Since by the $\alpha$-expanding assumption the second fundamental form is nonpositive $h\leq 0$, then $H=0$ is equivalent to $h=0$. Thus either $\alpha=\frac13$ or $h=0$ or $|\W|=0$. Moreover, if $|\W|=0$ (and $H\neq 0$), then \eqref{eq-eqeq3} implies $\D_{T}\S=0$ and the assumption $\D_{T}\S^{\bf{pf}}=0$ gives also 
$$
0=\D_T \left(\mathbf{s_{crit}}\right) = H \,\mathbf{s_{crit}}\,,
$$
i.e. $\mathbf{s_{crit}}=0$ and $\alpha=\frac13$. 
\end{proof}
%
%
%As a corollary we obtain the following (weak) monotonicity formula for $\S$
%
%\begin{corollary}\label{cor-main} Let $\Pa$ be a $k$-perfect fluid electric $\alpha$-expanding region such that
%$$
%k>(1-3\alpha)\,.
%$$
%Then, there exists a positive constants $m=m(k,\alpha)$ such that, if $M \geq m \L$ on $\Pa$, then
%$$
%\D_T \left(\S+\mathbf{s_{crit}}\sqrt{g}\right) \geq 0\,,
%$$
%with
%$$
%\mathbf{s_{crit}}=\frac{2k(1-3\alpha)}{\sqrt{9k^2-24k+20}}\,.
%$$
%Moreover, equality holds at some point if and only if $h=0$. 
%\end{corollary}
%\begin{proof} 
%
%The inequality follows immediately from Proposition \ref{pr-entest}. ***
%
%Clearly, if $h=0$ at some point, the by \eqref{eq-grapf}, \eqref{eq-evA} and \eqref{eq-evdet} we get that the quantity $\S+\mathbf{s_{crit}}\sqrt{g}$ must be constant in time. On the other hand, if equality holds in the estimate of Corollary \ref{cor-main}, then from equation \eqref{eq-1001} we get that either $H=0$ or $\L=0$ (recall that the mass density $M$ has to be positive, when $\L=0$, in order $\S$ to be defined). Moreover, the right hand side of the estimate in Proposition \ref{pr-entest} must vanish. Thus, if $\L=0$, since $M>0$ and $k>1-3\alpha$ we obtain $H=0$. In particular, $h_{ij}=0$, since by the $\alpha$-expanding assumption $h_{ij}\leq 0$. 
%\end{proof}
%

\begin{remark} The critical Weyl entropy density 
$$
\mathbf{s_{crit}}:=\sqrt{1-3\alpha}\left(\frac{\sqrt{2}}{4}+2k\sqrt{\frac{1-3\alpha}{9k^2-12k+8}}\right)
$$
in the following special eras of cosmic evolution is given by

\begin{itemize}

\item {\em Radiation Dominated:} $k=\frac43$

$$
\mathbf{s_{crit}}=\sqrt{1-3\alpha}\left(\frac{\sqrt{2}}{4}+\frac{2\sqrt{2}}{3}\sqrt{1-3\alpha}\right)\,.
$$

\item {\em Matter Dominated:} $k=1$

$$
\mathbf{s_{crit}}=\sqrt{1-3\alpha}\left(\frac{\sqrt{2}}{4}+\frac{2}{\sqrt{5}}\sqrt{1-3\alpha}\right)\,.
$$

\item {\em Vacuum Energy Dominated:} $k=0$

$$
\mathbf{s_{crit}}\sim\sqrt{1-3\alpha}\left(\frac{\sqrt{2}}{4}\right)\,.
$$

\end{itemize}
\end{remark}

\

\section{A Special Class of Pure Electric Spacetimes}\label{s-gencase}

Let $(\X,\gb)$ be a spacetime where the metric takes the form
$$
\gb=-N^2(x,t)dt^2+g_{ij}(x,t)dx^i dx^j 
$$
with $N>0$ and
\begin{equation}\label{eq-qstat}
g_{ij}(t,x)=e^{2\sigma(t,x)} \bar{g}_{ij}(x)
\end{equation}
for a given positive function $\sigma:I\times M^3\to\RR$. In this case the foliation is totally umbilical, i.e.
\begin{equation}\label{eq-umb}
h_{ij} = \frac{H}{3}g_{ij}\,.
\end{equation}
From equations \eqref{eq-ricvac}, we obtain
\begin{align*}
 N R_{ij} &= \nabla_i \nabla_j N + \left(N \L+\frac13 \dot{H}-\frac19N H^2\right) g_{ij} \\
0 &= \nabla_j H-\nabla_k h_{jk} \\
\Delta N &= -\L N - \dot{H} - \frac13 N H^2 \,.
\end{align*}
From the second equation and \eqref{eq-umb} we obtain
$$
0 = \nabla_j H-\nabla_k h_{jk} = \frac{2}{3} \nabla_j H \,,
$$
i.e. the foliation has constant (in space) mean curvature $H\equiv H(t)$ on $M^3$. In particular the second fundamental form $h_{ij}$ is a Codazzi tensor and the spacetime $(\X,\gb)$ satisfies \eqref{H1}, if the stress-energy tensor is diagonal. Thus, if we consider a $k$-perfect fluid spacetime $(\X,\gb)$ of this form, we have that $(\X,\gb)$ is a $k$-perfect fluid electric $\frac13$-expanding region, if we assume $H\leq 0$ or, equivalently, $\p_t \sigma \geq 0$. In this case, since
$$
\mathbf{s_{crit}}=0,
$$
we have $\S^{\bf{pf}}=\S$ and Theorem \ref{t-mon} implies
\begin{proposition} On a $k$-perfect fluid expanding spacetime satisfying \eqref{eq-qstat} the   entropy $\S$ is monotonically increasing.
\end{proposition}

\

\section{Perfect Fluid Magnetic Regions}

In this section we will consider spacetime $(\X,\gb)$, satisfying
\begin{equation}\label{H2}\tag{H2}
\W_{TiTj}\equiv 0\quad\text{on } X\,.
\end{equation}
As already observed, spacetimes satisfying \eqref{H2} are called {\em Pure Magnetic Spacetime}. Localizing this notion we provide the following 

\begin{definition}\label{d-pfm}  Let $(\X,\gb)=(I\times M^3, -N^2 dt^2+g)$ be an almost globally hyperbolic spacetime, let $U\subset M^3$ and $I'\subset I$ be two empty open sets. We say that $\Pm:=I'\times U$ is a {\em $k$-perfect fluid magnetic region} if 
\begin{itemize}

\item[(1)] the stress-energy tensor $\T$ is given by \eqref{eq-pf} on $\Pm$, with $P=(k-1)M$,  $k\in[0,\frac43]$ and 
$$0\leq k' \leq 
\frac{k(9k^2-12k+8)}{3(4-3k)},$$ where $\D_T k=k'H$ (where $H\neq 0$). 

\item[(2)] the Weyl tensor $\W$ satisfies \eqref{H2} on $\Pm$;

\item[(3)] the lapse function $N$ is strictly positive on $\Pm$.

\end{itemize}

Moreover, if there exists a function $\alpha=\alpha(t,x)\in\left[0,\frac13\right]$ such that
$$
h_{ij} \leq \alpha H g_{ij} \leq 0 \quad \text{on}\quad I'\times U\,,
$$
we will say that $\Pma$ is a $k$-perfect fluid magnetic $\alpha$-expanding region.
\end{definition}

\begin{lemma}\label{l-evM} Let $(\X,\gb)$ be a globally hyperbolic spacetime satisfying \eqref{H2}. Then 
\begin{align*}
\frac12\D_T|\W_{\a\b\g\d}|^2 =-16h_{kp}\W_{Tijk}\W_{Tijp}\,.
\end{align*}
\end{lemma}

\begin{proof}
From the second bianchi identity \eqref{eq-sbinv} 
\begin{align*}
\D_T \W_{Tijk}+\D_{j} \W_{TikT}+\D_{k} \W_{TiTj} &= \frac12 \left(\C_{TkT}\gb_{ij}+\C_{TTj}\gb_{ik}+\C_{Tjk}\gb_{iT}\right) \\ 
&\quad-\frac12 \left(\C_{ikT}\gb_{Tj}+\C_{iTj}\gb_{Tk}+\C_{ijk}\gb_{TT}\right) \\
&=\frac12 \left(\C_{TkT}g_{ij}-\C_{TjT}g_{ik}+\C_{ijk}\right)\,.
\end{align*}
Also, \eqref{eq-cot} and the fact that $\D_i M = 0$ imply
\begin{align*}
\C_{TkT} &:=  \D_{T}\T_{kT} - \D_{k}\T_{TT}-
\frac{1}{3}  \left( \D_{T}\T  \gb_{kT} -  \D_{k}\T
\gb_{TT} \right) \\
&= -\frac 13 \D_k \T = \frac{4-3k}{3}\D_k M = 0 \,
\end{align*}
and
\begin{align*}
\C_{ijk} &:=  \D_{k}\T_{ij} - \D_{j}\T_{ik}-
\frac{1}{3}  \left( \D_{k}\T  \gb_{ij} -  \D_{j}\T
\gb_{ik} \right) \\
&= \D_{k}\T_{ij} - \D_{j}\T_{ik} =0\,.
\end{align*}
Thus
$$
\D_T \W_{Tijk} = - \D_{j} \W_{TikT}-\D_{k} \W_{TiTj}\, = \D_{j} \W_{TiTk}- \D_{k} \W_{TiTj}\,.
$$
Moreover, condition \eqref{H2} implies
\begin{align*}
\D_{k} \W_{TiTj} &= \p_{k}\W_{TiTj} - \Gb_{kT}^p \W_{piTj}-\Gb_{kT}^{p}\W_{Tipj}\\
&=  - \Gb_{kT}^p \W_{Tjpi}-\Gb_{kT}^{p}\W_{Tipj} \\
&=h_{kp}(\W_{Tjpi}+\W_{Tipj})\,.
\end{align*}
Rewriting the last equations we have proved that
\begin{equation*}
\D_T \W_{Tijk} = h_{jp}(\W_{Tkpi}+\W_{Tipk})-h_{kp}(\W_{Tjpi}+\W_{Tipj})\,.
\end{equation*}
Thus, from Lemma \ref{l-for}, we get
\begin{align*}
\frac12\D_T|\W_{\a\b\g\d}|^2 &= -2 \D_T |\W_{Tijk}|^2 = -4 \W_{Tijk} \D_T \W_{Tijk} \\
&= -4 h_{jp}(\W_{Tkpi}+\W_{Tipk})\W_{Tijk}+4h_{kp}(\W_{Tjpi}+\W_{Tipj})\W_{Tijk} \\
&= 8 h_{kp}(\W_{Tjpi}+\W_{Tipj})\W_{Tijk}\\
&=-16h_{kp}\W_{Tijk}\W_{Tijp}\,.
\end{align*}

\end{proof}

\begin{proposition}\label{c-entestM} Let $(\X,\gb)$ be a globally hyperbolic perfect fluid satisfying \eqref{H2}. Then the   Weyl entropy $\S$ satisfies
\begin{align*}
\D_T \S &= \frac{|\A|^3}{|\R|_{\gbr}^3}\Big[ 16 \frac{h_{jl}\W_{Tijk}\W_{Tijp}}{|\W||\A|}\\ &\quad+H\frac{|\W|}{3|\A|^3}\left[k(9k^2-12k+8)+3k'(3k-2)\right]M^2+H \frac{|\W||\R|_{\gbr}^2}{|\A|^3}\Big]\sqrt{g} 
\end{align*}
\end{proposition}
\begin{proof} Using \eqref{ge-wge} and \eqref{eq-nw} we obtain
$$
\D_T \S = \D_{T}\left(\mathbf{s}\sqrt{g}\right)=\left(\frac{\mathbf{s}}{\bf{\bar{s}}}\right)^3\D_T \mathbf{\bar{s}}\,\sqrt{g} + \mathbf{s}\,\D_T \sqrt{g} \,.
$$
Since, from Lemma \ref{l-for} $|\A|_{\gbr}^2=|\A|^2$ and, from \eqref{eq-rm2pf} and \eqref{eq-evA}, we have
$$
|\A|^2=\frac{9k^2-12k+8}{3}M^2
$$
and
$$
\D_T |\A| = \frac{k(9k^2-12k+8)+3k'(3k-2)}{3|\A|} H\,M^2\,.
$$
Thus, from \eqref{eq-sss}, Lemma \ref{l-for} and Lemma \ref{l-evM} we obtain
\begin{align*}
\D_T \bf{\bar{s}}&=-\D_T \left(\frac{|\W|}{|\A|}\right) = -\frac{1}{|\A|^2}\left(|\A| \D_T |\W| - |\W| \D_T |\A|\right) \\
&=-\frac{1}{|\A|^2} \Big[ \frac{|\A|}{|\W|}\left(-16h_{kp}\W_{Tijk}\W_{Tijp} \right)\\ &\quad-H\frac{|\W|}{3|\A|}\left[k(9k^2-12k+8)+3k'(3k-2)\right]M^2\Big] \\
&= 16 \frac{h_{jl}\W_{Tijk}\W_{Tijp}}{|\W||\A|}\\ &\quad+H\frac{|\W|}{3|\A|^3}\left[k(9k^2-12k+8)+3k'(3k-2)\right]M^2 \,.
\end{align*}
From \eqref{eq-evdet}, one has
\begin{align*}
\D_T \sqrt{g} =  - H \sqrt{g}\,.
\end{align*}
Using Lemma \ref{l-for}, we get
\begin{align*}
\D_T \S &= \frac{|\A|^3}{|\R|_{\gbr}^3}\Big[ 16 \frac{h_{jl}\W_{Tijk}\W_{Tijp}}{|\W||\A|}\\ &\quad+H\frac{|\W|}{3|\A|^3}\left[k(9k^2-12k+8)+3k'(3k-2)\right]M^2+H \frac{|\W||\R|_{\gbr}^2}{|\A|^3}\Big]\sqrt{g} 
\end{align*}
and this concludes the proof.
\end{proof}

As a corollary, we have the following monotonicity of $\S$ on perfect fluid magnetic $\alpha$-expanding regions. 

\begin{theorem}\label{teo-mainM} Let $\Pma$ be a $k$-perfect fluid magnetic $\alpha$-expanding region. Then,
$$
\D_T \S \geq 0\,.
$$
Moreover, the equality holds at some point if and only  if either $h=0$ or $|\W|=0$. 
\end{theorem}
\begin{proof} From Proposition \ref{c-entestM} we have
\begin{align*}
\D_T \S &= \frac{|\A|^3}{|\R|_{\gbr}^3}\Big[ 16 \frac{h_{jl}\W_{Tijk}\W_{Tijp}}{|\W||\A|}\\ &\quad+H\frac{|\W|}{3|\A|^3}\left[k(9k^2-12k+8)+3k'(3k-2)\right]M^2+H \frac{|\W||\R|_{\gbr}^2}{|\A|^3}\Big]\sqrt{g} 
\end{align*}
Since $|\W|\leq 0$, $h_{ij}\leq \alpha H g_{ij}\leq 0$, $|\R|_{\gbr}\geq |\A|$,
$$
0\leq k' \leq \frac{k(9k^2-12k+8)}{3(4-3k)}
$$
and $\D_{T}N\geq 0$ we obtain
\begin{align*}
\D_T \S &\geq \frac{|\A|^3}{|\R|_{\gbr}^3}\Big[ (1+4\alpha) H \frac{|\W|}{|\A|}\\ &\quad+H\frac{|\W|}{3|\A|^3}\left[k(9k^2-12k+8)+3k'(3k-2)\right]M^2\Big]\sqrt{g} \\
&\geq \frac{|\A|^3}{|\R|_{\gbr}^3}\Big[ (1+4\alpha) H \frac{|\W|}{|\A|}\Big]\sqrt{g} \geq 0\,,
\end{align*}
since $\alpha\geq 0$ by assumption. The equality case follows easily. 
\end{proof}

\

\section{Regions with Maximum or Minimal Weyl Entropy}\label{s-max}

\begin{figure}\label{fig-2}
\centering
\includegraphics[scale=0.7]{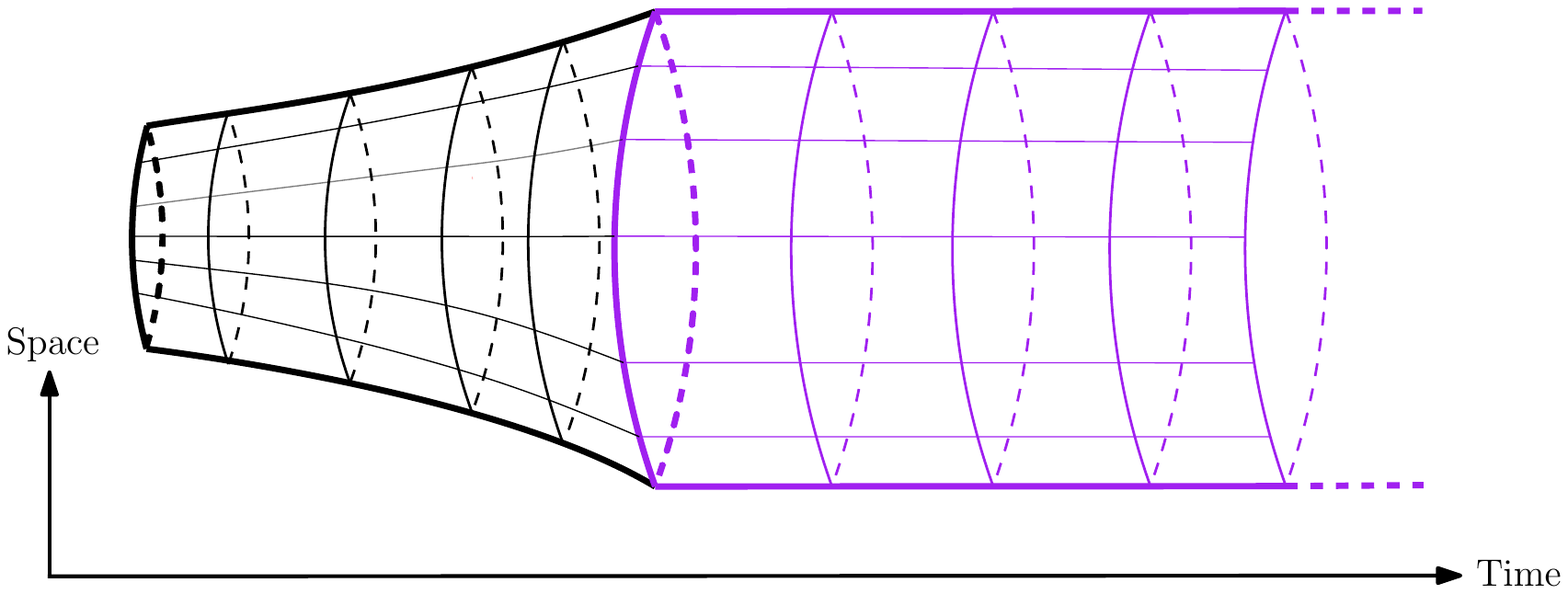}
\caption{A picture of a region with {\em maximum Weyl Entropy} which become static.}
\end{figure}

Let $I\times U$, be a $k$-perfect fluid electric region contained in an almost globally hyperbolic spacetime $(\X, \gb)$ and satisfying \eqref{eq-asspar}. From Theorem \ref{t-mon}, for every $t\in I$ we have
$$
\D_T\S^{\bf{pf}}_{U}=\frac{\text{Area}(\Sigma)}{\operatorname{Vol}_{g}(U)}\int_{U}\D_T\S^{\bf{pf}}+\D_T\left(\frac{\text{Area}(\Sigma)}{\operatorname{Vol}_{g}(U)}\right)\int_{U}\S^{\bf{pf}}\geq \D_T\left(\frac{\text{Area}(\Sigma)}{\operatorname{Vol}_{g}(U)}\right)\int_{U}\S^{\bf{pf}}.
$$
In particular, to guarantees the monotonicity of the Weyl entropy in $U$, we can assume the following 
\begin{equation}\label{eq-assU}
\D_T\left(\frac{\text{Area}(\Sigma)}{\operatorname{Vol}_{g}(U)}\right)\geq 0.
\end{equation}

\subsection{Maximal case.} We say that a time slice $U_{t_0}:=\{t_0\}\times U$, $t_0\in I$, has {\em maximum Weyl entropy} at time $t_0$, if the   Weyl entropy $\S^{\bf{pf}}$ in $U$ is maximal at time $t_0\in I$. We recall that
$$
\S^{\bf{pf}}_{U}:=\frac{\text{Area}(\Sigma)}{\operatorname{Vol}_{g}(U)}\int_{U}\S^{\bf{pf}}=\frac{\text{Area}(\Sigma)}{\operatorname{Vol}_{g}(U)}\int_{U}\left(\frac{|\W|_{\gbr}}{|\R|_{\gbr}}+\mathbf{s_{crit}}\right)\sqrt{g}\,,
$$
if $\R\neq 0$, or 
$$
\S^{\bf{pf}}_{U}=\text{Area}(\Sigma)\left(1+\frac{1}{\operatorname{Vol}_{g}(U)}\int_{U}\mathbf{s_{crit}}\sqrt{g}\right)
$$ 
otherwise. Since $|\W|_{\gbr}\leq |\R|_{\gbr}$ with equality if and only if $\A=0$ (i.e. the Ricci tensor of $\gb$ vanishes), we get
$$
\S^{\bf{pf}}_{U} \leq \text{Area}(\Sigma)\left(1+\frac{1}{\operatorname{Vol}_{g}(U)}\int_{U}\mathbf{s_{crit}}\sqrt{g}\right)\leq \text{Area}(\Sigma)\left(1+ \sup_U \mathbf{s_{crit}}\right) 
$$
with equality if and only if the Ricci tensor $\R_{\a\b}$ vanishes (the case of zero Riemann tensor $\R$ is automatically included) and $\mathbf{s_{crit}}$ is constant on $U$.

Assume now that $U_{t_0}$ has maximum Weyl entropy at time $t_0$. By the monotonicity in Theorem \ref{t-mon}, all the future slices $U_t$ have maximum entropy, for all $I\ni t\geq t_0$. In this case, for all $I\ni t\geq t_0$, the Ricci tensor is zero on $U_{t}$, $\mathbf{s_{crit}}$ is constant on $U_t$ and, from the Einstein equation, the stress-energy tensor $\T$ is zero, i.e. 
$$
\R_{\a\b}=0 \quad\text{and}\quad \T_{\a\b}=0\quad\text{on }U_t\,.
$$

\begin{figure}\label{fig-3}
\centering
\includegraphics[scale=0.6]{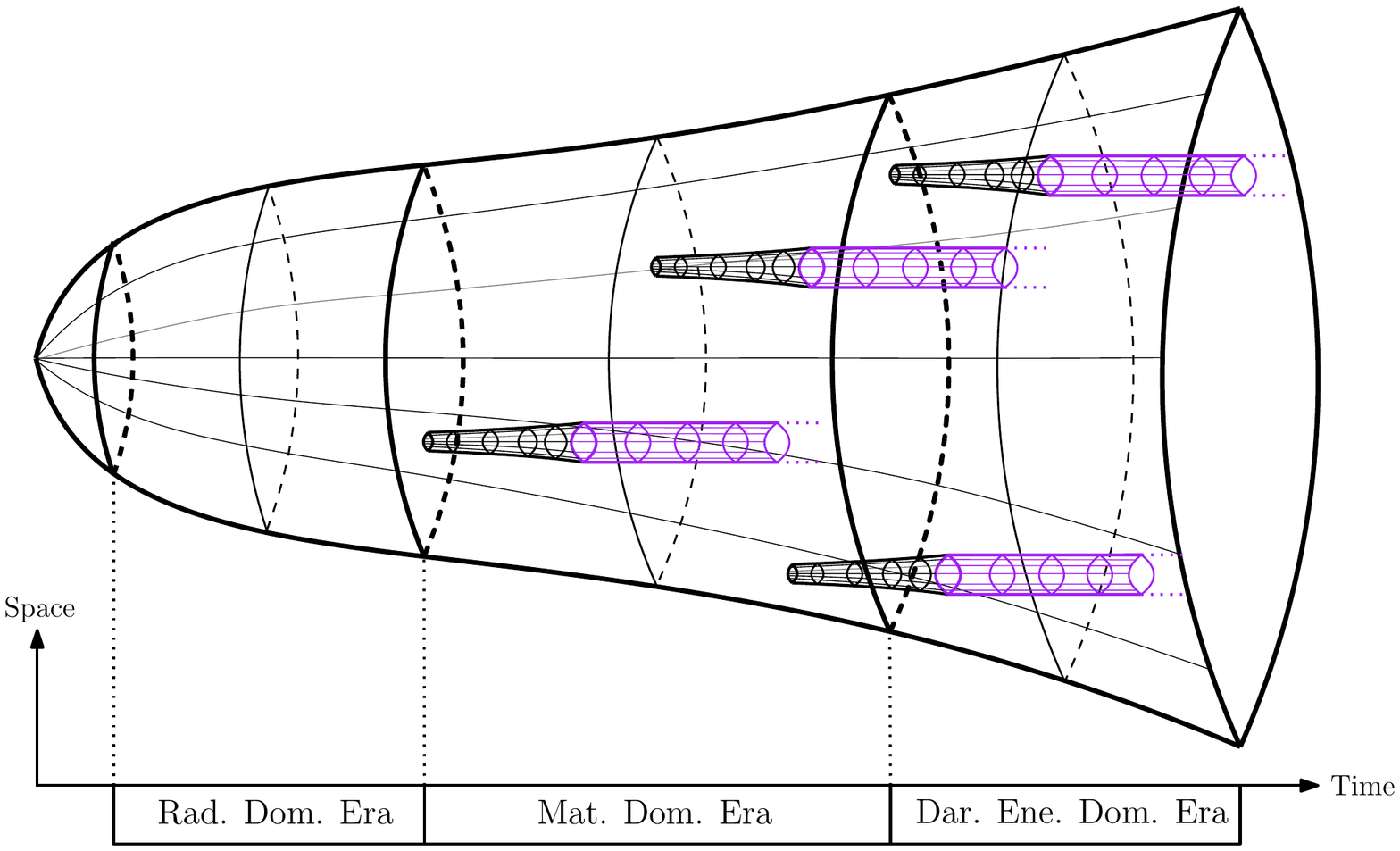}
\caption{A picture of a cluster of regions with {\em maximum Weyl Entropy} which become static.}
\end{figure}

Moreover, by maximality, from the equality case in the monotonicity Theorem \ref{t-mon}, we obtain that, on  $U_{t}$, one has $\D_T \mathbf{s_{crit}}=0$ (and $\mathbf{s_{crit}}=\sup_U \mathbf{s_{crit}}$) and either $h=0$ or $|\W|=0$ and $\alpha=\frac13$ (and $\mathbf{s_{crit}}=0$). In the second case, since it is Ricci flat, the region is flat, i.e. locally isometric to a quotient of the Minkowski Space. In the first case, $h=0$, the space metric $g$ does not depend on $t$ and from equations \eqref{eq-ricvac}, the region $U_{t}$ is {\em static}. More precisely, the triple $(U,g,N)$ satisfies the following system
$$
\begin{cases}
N\,R_{ij} = \nabla_i\nabla_j N  \\
\Delta N = 0 
\end{cases} \quad\text{on } U\,.
$$
In particular $(U, g)$ has zero scalar curvature $R=0$. Note that we can ensure that the lapse function $N$, which a priori could depend on time, is in fact equal to $N(t_0,x)$ on $U_t$, for every $I\ni t\geq t_0$. This follows immediately form the local uniqueness for the Cauchy development satisfying the constraint equations (which in this case is simply $R=0$) on the slice $\{t_0\}\times U$. In both cases, the solution is a static vacuum spacetime in the whole $[t_0,+\infty)\cap I \times U$ (see Figure 2 and Figure 3) and the Weyl entropy $\S^{\bf{pf}}_{U}$ is constant  in time and equal to
$$
\S^{\bf{pf}}_{U}=\text{Area}(\Sigma)\left(1+ \sup_U \mathbf{s_{crit}}\right) \,.
$$
Suppose now that $t\in[t_0,T)$ for some maximal time $T\in \RR^*$. Then, either $T=+\infty$ (and in this case the lapse function must be strictly positive everywhere) or $T<+\infty$. In this second case, assume that there exists a point $x\in \bar{U}_t$ such that $\lim_{t\to T} N(t, x)=0$. Since $N(t,x)=N(t_0,x)$, this implies that $t_0$ actually has to be the maximal time of existence $T$. Since $N(t_0,\cdot)>0$ in $U$, the point $x\in\Sigma=\partial U$. By Hopf lemma, since the lapse function $N$ is nonnegative on $U_T^c$, the function $N$ has to be identically zero on the boundary $\Sigma$, i.e.
$$
N(x)=0\quad\text{for all}\,\,\, x\in \Sigma \,.
$$ 
This means that, either the maximum entropy region is the whole cylinder $[t_0,+\infty)\times U$ or, $T<\infty$ and either $N>0$ or $t_0=T$ and the lapse function $N$ vanishes on the surface $\Sigma$. In this case $\Sigma$ is the so called a {\em event horizon} and we have the maximal Weyl entropy satisfies
$$
\S^{\bf{pf}}_{U}=\text{Area}(\Sigma)\left(1+ \sup_U \mathbf{s_{crit}}\right)=\S_{\text{GH}}\left(1+ \sup_U \mathbf{s_{crit}}\right)
$$
where $\S_{\text{GH}}$ is governed by the same area law as the black hole entropy but is associated with a cosmological event horizon, and was introduced by Gibbons and Hawking \cite{gibhaw}.  It is still a maximal entropy in itself, and is associated with a sort of equilibrium state of the geometry, i.e. with a static solution.

Reasoning as above, the same things for the Weyl entropy  $\S_U$ happen in a maximal $k$-perfect fluid magnetic region, as a consequence of Theorem \ref{teo-mainM}. We note that a static region cannot occur, since in case the Weyl tensor would be pure electric. Therefore, maximal Weyl entropy regions must be flat (vacuum).

\subsection{Minimal case.} We deal now with the minimal case. We say that a time slice $U_{t_0}=\{t_0\}\times U$, $t_0\in I$, has {\em minimal Weyl entropy} at time $t_0$, if the Weyl entropy $\S^{\bf{pf}}$ in $U$ is minimal at time $t_0\in I$. Since $\S^{\bf{pf}}_{U}\geq 0$, the minimal value is zero, from $|\W|_{\gbr}\geq 0$ and $\mathbf{s_{crit}}\geq 0$, we obtain that at 
\begin{equation}\label{eq-cmin}
\W\equiv 0,\quad \mathbf{s_{crit}}\equiv 0\quad\text{on }\,U_{t_0}\quad\Longleftrightarrow\quad \W\equiv 0,\quad \alpha\equiv\frac13\quad\text{on }\,U_{t_0}
\end{equation}
Assume now that $U_{t_0}$ has minimal Weyl entropy at time $t_0$. By the monotonicity in Theorem \ref{t-mon}, all the past slices $U_t$, have minimal entropy, for all $I\ni t\leq t_0$. In this case, for all $I\ni t\leq t_0$, the Weyl tensor vanishes and $\alpha=\frac13$ on $U_{t}$. We claim that on the spacetime cylinder $\left([0,t_0]\cap I\right)\times U$ the metric must be of Friedmann-Lema\^{i}tre-Robertson-Walker type, i.e. the spacetime metric has the form
$$
\gb=-N(t)^2\,dt^2+a(t)^2\,g^K_{ij}(x)dx^i\,dx^j,
$$
where $N=N(t)$ is the original lapse function which now depends only on time, $a=a(t)$ is a positive function depending only on time and $g^K$ is a Riemannian metric of constant sectional curvature $K$ on $U$. It is a FLRW metric up to rescaling the time appropriately with the lapse function. In fact, from \eqref{eq-cmin} and the definitions of $\alpha$-expansion \eqref{d-pfe} and \eqref{d-pfm}, we get that the slices must be umbilical
$$
h_{ij}=\frac13 H\,g_{ij}.
$$
In particular, from \eqref{eq-ricvac}, we obtain that 
$$
0 = \nabla_j H-\nabla_k h_{jk} = \frac{2}{3} \nabla_j H \,,
$$
i.e. the foliation has constant (in space) mean curvature $H\equiv H(t)$ on $U$. Thus, from umbilicity and \eqref{eq-sf}, we get 
\begin{equation}\label{eq-dth}
\partial_t h_{ij} = \frac13 (\partial_t H)g_{ij}+\frac13 H\partial g_{ij}=\frac19\left(\partial_t H-2NH^2\right)g_{ij}.
\end{equation}
Since $\T_{TT}=M$, $\T_{ij}=P g_{ij}=(k-1)M g_{ij}$ and $\W\equiv 0$, from \eqref{eq-g3} and \eqref{eq-dth}, we obtain
$$
0 = N\W_{TiTj}=-N\frac{2-k}{2}M g_{ij}+N\frac{4-3k}{3}M g_{ij} +\frac19\left(\partial_t H-2NH^2\right)g_{ij}+\frac{1}{9}NH^2 g_{ij}+\nabla_i\nabla_j N,
$$
and thus
$$
\nabla_i\nabla_j N = \varphi(t,x) g_{ij}
$$
for some function $\varphi$. Using this equation and \eqref{eq-dth} in \eqref{eq-ricvac}, we obtain
$$
N\frac{2-k}{2}M g_{ij} = N\,R_{ij}-\frac19\left(\partial_t H-2NH^2\right)g_{ij}+\frac{1}{3}NH^2 g_{ij}-\frac{2}{9}NH^2 g_{ij}-\varphi g_{ij}
$$
and therefore, the Ricci tensor of the metric $g$ satisfies in $U$
$$
R_{ij}=\psi(t,x)g_{ij}
$$
for some function $\psi$. By Bianchi identity, the function $\psi$ must be constant in space, $\psi=\psi(t)$, and the metric $g$ is Einstein. Since $U\subset M^3$ is three-dimensional, $(U,g)$ must have constant (in space) sectional curvature. Thus we have proved that 
$$
g_{ij}(t,x)=a(t)^2\,g^K_{ij}(x)
$$
where $g^K$ is a Riemannian metric of constant sectional curvature $K$ on $U$. Thus, the spacetime metric $\gb$ takes the form
$$
\gb=-N(t,x)^2\,dt^2+a(t)^2\,g^K_{ij}(x)dx^i\,dx^j.
$$
It remains to show that the lapse function is constant in space, $N=N(t)$. By the conformal flatness of $\gb$, $\W\equiv 0$, we have that the conformal metric
$$
\widetilde{\gb}=N(t,x)^{-2}\gb=-dt^2+\left[\frac{a(t)}{N(t,x)}\right]^2\,g^K_{ij}(x)dx^i\,dx^j.
$$
is conformally flat. From well known results in Riemannian (or, more in general, pseudo-Riemannian) Geometry, it is possible to show (by direct local computation) that the only possibility for $\widetilde{\gb}$ to be conformally flat, is that the warping function $a/N$ depends only on $t$, i.e. the lapse function $N$ depends only on $t$. This concludes the proof of the claim. 

To conclude, we have shown that if a time slice $U_{t_0}=\{t_0\}\times U$ has minimal (zero) Weyl entropy, then the spacetime cylinder $\left(\left([0,t_0]\cap I\right)\times U, \gb\right)$ is a FLRW spacetime. Again, by uniqueness for the Cauchy development, this must happen on the entire cylinder $I\times U$, if the FLRW solution exists until the final time $T\in I$. 

As this situation is unphysical, we can avoid it by assuming the condition $\D_T \alpha(0)<0$ at some point $x\in U$, which, under our hypotheses, 
would imply  $\alpha(t,x)<\frac13$ for all $t\in I$. This hypothesis is quite mild, as it correspond to assuming that the Universe is non-homogeneous 
for $t>0$. 

\

\section{Conclusions}

We have taken into account the longstanding  Penrose's Weyl Curvature Hypothesis, concerning the expression and the time evolution of the so-called gravitational entropy. We have privileged a general framework which is able to  encompass specific models discussed in literature. On the one hand, the Weyl entropy density $\S$ and the related Weyl entropy for the case of a perfect fluid, $ \S^{\bf{pf}} := \S+ \mathbf{s_{crit}} \sqrt{g}$ have been introduced, and we have shown that they are  monotonically increasing in time under very general assumptions, 
for the so-called pure spacetime regions, i.e. regions purely electric/magnetic, with reference to the 
standard decomposition of the Weyl tensor. On the other hand, 
our ansatz for the Weyl entropy of compact regions $\S^{\bf{pf}}_{U}$ has been show to satisfy monotonicity in time too, under 
physically reasonable hypotheses. In our framework, we have also found that $\S^{\bf{pf}}_{U}$  is maximal 
for vacuum static metrics, which should then correspond to sort of equilibrium states of the gravitational field. 
As a further development of this work, one should look for a statistical mechanical derivation of the expression for the 
candidate Weyl entropy, and, furthermore, try to relax the requirement for pure spacetime regions. Both these tasks 
represent a very interesting field of investigation, to be deferred to future works.

\

\begin{ackn}
\noindent The second author is member of the Gruppo Nazionale per l'Analisi Matematica, la Probabilit\`{a} e le loro Applicazioni (GNAMPA) of the Istituto Nazionale di Alta Matematica (INdAM).
\end{ackn}

\

\appendix\section{}

We have the following formulas relating the Weyl tensor components.

\begin{lemma}\label{l-alg} Let $(\X,\gb)$ be a globally hyperbolic  spacetime. Then  
\begin{align*}
\W_{TiTj} &= g^{kl} \W_{kilj} \\
\W_{ijkl} &= \left(\W_{TiTk}\,g_{jl}-\W_{TiTl}\,g_{jk}+\W_{TjTl}\,g_{ik}-\W_{TjTk}\,g_{il}\right) \,.
\end{align*}
\end{lemma}
\begin{proof}
Since $\W$ is trace free,  one has 
$$
0 =\gb^{\a\b} \W_{\a i \b j} = - \W_{TiTj} + g^{kl} 	\W_{kilj} 
$$
and the first equality follows. To obtain the second equation we recall that from \eqref{eq-ricvac} and \eqref{eq-coneq} one has
\begin{align*}
N\left(\T_{ij}-\frac12\T g_{ij}\right) &= N R_{ij} -\p_t h_{ij} +N H h_{ij}-2N h_{il}h_{jl} -\nabla_i \nabla_j N\\
R + H^2 -|h|^2 &= 2\T_{TT}\,.
\end{align*}
Hence, from \eqref{eq-g3} we get
\begin{align}\label{eq-wt}
\W_{TiTj} = R_{ij} - \frac12 \T_{ij} - \frac{3R-2\T}{12} g_{ij}-\frac{H^2-|h|^2}{4}g_{ij}+Hh_{ij}-h_{il}h_{jl}\,.
\end{align}
Thus, from \eqref{eq-g1} one obtain
\begin{align*}
\W_{ijkl} &= \left(\W_{TiTk}\,g_{jl}-\W_{TiTl}\,g_{jk}+\W_{TjTl}\,g_{ik}-\W_{TjTk}\,g_{il}\right)\\
&\quad-\frac{|h|^2-H^2}{2}\left(g_{ik}g_{jl}-g_{il}g_{jk}\right)+ h_{ik}h_{jl}-h_{il}h_{jk} \\
&\quad +\left(h_{ip}h_{kp}g_{jl} -h_{ip}h_{lp}g_{jk}+h_{jp}h_{lp}g_{ik}-h_{jp}h_{kp}g_{il}\right)\\
&\quad - H\left(h_{ik}g_{jl}-h_{il}g_{jk}+h_{jl}g_{ik}-h_{jk}g_{il}\right)
\end{align*}
Setting
\begin{align*}
Z_{ijkl} &:= -\frac{|h|^2-H^2}{2}\left(g_{ik}g_{jl}-g_{il}g_{jk}\right)+ h_{ik}h_{jl}-h_{il}h_{jk} \\
&\quad +\left(h_{ip}h_{kp}g_{jl} -h_{ip}h_{lp}g_{jk}+h_{jp}h_{lp}g_{ik}-h_{jp}h_{kp}g_{il}\right)\\
&\quad - H\left(h_{ik}g_{jl}-h_{il}g_{jk}+h_{jl}g_{ik}-h_{jk}g_{il}\right)\,,
\end{align*}
Take a basis diagonalizing $h$, one has $h_{ij}=\mu_i \d_{ij}$ and for $i\neq j\neq k$ we get
\begin{align*}
Z_{ijij} &:= -\frac{|h|^2-H^2}{2} + \mu_i \mu_j + \mu_i^2 + \mu_j^2 - H(\mu_i+\mu_j)\\
&= \frac12 \left[(\mu_i+\mu_j+\mu_k)^2-(\mu_i^2+\mu_j^2+\mu_k^2) \right]+ \mu_i \mu_j + \mu_i^2 + \mu_j^2 - (\mu_i+\mu_j+\mu_k)(\mu_i+\mu_j) \\
&= 2\mu_i \mu_j + \mu_i\mu_k + \mu_j\mu_k + \mu_i^2 + \mu_j^2 - (\mu_i+\mu_j+\mu_k)(\mu_i+\mu_j) = 0
\end{align*}
Thus $Z\equiv 0$ and the proof is completed.
\end{proof}

As a consequence, we obtain the following identities.

\begin{lemma}\label{l-for} Let $(\X,\gb)$ be a globally hyperbolic spacetime. Then the following identity hold
\begin{align*}
|\W_{\a\b\g\d}|^2 &=-4|\W_{Tijk}|^2+8|\W_{TiTj}|^2,\\
|\A_{\a\b\g\d}|^2 &=-4|\A_{Tijk}|^2+4|\A_{TiTj}|^2+|\A_{ijkl}|^2\\
|\R_{\a\b\g\d}|_{\gbr}^{2}&=|\W_{\a\b\g\d}|_{\gbr}^{2}+|\A_{\a\b\g\d}|_{\gbr}^{2}
%\,,\\
%\W_{\a\b\g\d}\W_{\a\b\e\n}\W_{\g\d\e\n} &= -12\W_{TiTk}\W_{Tipq}\W_{Tkpq} -16\W_{TiTk}\W_{TiTq}\W_{TkTq}\,.
\end{align*}
and, if $\A_{Tijk}=0$, then
$$
|\A_{\a\b\g\d}|_{\gbr}^{2}=|\A_{\a\b\g\d}|^{2}\,.
$$
In particular, if $(\X,\gb)$ satisfies \eqref{H1}, then
\begin{align*}
|\W_{\a\b\g\d}|^2 &=8|\W_{TiTj}|^2.
\end{align*}

%\,,\\
%\W_{\a\b\g\d}\W_{\a\b\e\n}\W_{\g\d\e\n} &= -16\W_{TiTk}\W_{TiTq}\W_{TkTq}\,.
\end{lemma}
\begin{proof}
One has
$$
|\W_{\a\b\g\d}|^2=-4|\W_{Tijk}|^2+4|\W_{TiTj}|^2+|\W_{ijkl}|^2\,.
$$
Using Lemma \ref{l-alg} we get
$$
|\W_{ijkl}|^2 = 4 |\W_{TiTj}|^2\,,
$$
and the first identity follows. The second identity follows for the same reason. By the definition of $\gbr$, we immediately have
\begin{align*}
|\R_{\a\b\g\d}|_{\gbr}^2&=4|\R_{Tijk}|^2+4|\R_{TiTj}|^2+|\R_{ijkl}|^2\\
|\W_{\a\b\g\d}|_{\gbr}^2&=4|\W_{Tijk}|^2+4|\W_{TiTj}|^2+|\W_{ijkl}|^2,\\
|\A_{\a\b\g\d}|_{\gbr}^2 &=4|\A_{Tijk}|^2+4|\A_{TiTj}|^2+|\A_{ijkl}|^2
\end{align*}
and the last formula holds.
\end{proof}

Concerning the covariant derivative $\D_{\e} \W_{\a\b\g\d}$ we have the following useful formulas which hold in Pure Electric Spacetime:

\begin{lemma}\label{l-gra} Let $(\X,\gb)$ be a globally hyperbolic spacetime satisfying \eqref{H1}. Then 
\begin{align*}
%|\D_\e \W_{\a\b\g\d}|^{2} &=-|\D_T \W_{ijkl}|^2-4|\D_p \W_{Tijk}|^2+4|\D_T \W_{Tijk}|^2\\
%&\quad +4|\D_p \W_{TiTj}|^2 -4|\D_T \W_{TiTj}|^2+|\D_p \W_{ijkl}|^2 \,,\\
\D_{T} \W_{ijkl} &= 2\left( h_{ik} \W_{TjTl} - h_{jk} \W_{TiTl}+h_{jl} \W_{TiTk}-h_{il} \W_{TjTk}\right) \\
&\quad+h_{lp}\left(\W_{TjTp}\,g_{ik}-\W_{TiTp}\,g_{jk}\right) +h_{kp}\left(\W_{TiTp}\,g_{jl}-\W_{TjTp}\,g_{il}\right)\\
&\quad +\frac12\left(\C_{ilT}g_{jk}+\C_{iTk}g_{jl}-\C_{jlT}g_{ik}+\C_{jTk}g_{il}\right)\,,\\
\D_{l} \W_{ijkT} &= 2\left(h_{jl} \W_{TiTk}-h_{il} \W_{TjTk}\right)+h_{lp}\left(\W_{TjTp}\,g_{ik}-\W_{TiTp}\,g_{jk}\right)  \,,\\
\D_T \W_{TiTj} &= 2H \W_{TiTj} - 2h_{il}\W_{TjTl}-h_{jl}\W_{TiTl}+h_{kp}\W_{TkTp}\,g_{ij}- \frac12 \C_{ijT}\,.
\end{align*}
\end{lemma}

\begin{proof}
From the second bianchi identity \eqref{eq-sbinv} we have
\begin{align*}
\D_T \W_{ijkl} &= - \D_{k} \W_{ijlT}-\D_{l} \W_{ijTk} +\frac12\left(\C_{ilT}g_{jk}+\C_{iTk}g_{jl}-\C_{jlT}g_{ik}+\C_{jTk}g_{il}\right) \\
&= \D_{l} \W_{ijkT}- \D_{k} \W_{ijlT} +\frac12\left(\C_{ilT}g_{jk}+\C_{iTk}g_{jl}-\C_{jlT}g_{ik}+\C_{jTk}g_{il}\right)\,.
\end{align*}
Moreover, from the definition of covariant derivative and condition \eqref{H1} implies
\begin{align*}
\D_{l} \W_{ijkT} &= \p_{l}\W_{ijkT} - \Gb_{li}^T \W_{TjkT}-\Gb_{lj}^T\W_{iTkT}-\Gb_{lk}^{\a}\W_{ij\a T}-\Gb_{lT}^p \W_{ijkp}\\
&=  \Gb_{li}^T \W_{TjTk}-\Gb_{lj}^T\W_{TiTk}-\Gb_{lT}^p \W_{ijkp} \\
&= -h_{il}\W_{TjTk}+h_{jl}\W_{TiTk}+h_{lp}\W_{ijkp}\,,
\end{align*}
where we have used the formulas 
\begin{align}\label{eq-chr} \nonumber
\Gb_{tt}^{t} &= -\frac{1}{2N^2}\p_t (N^2) = -\frac{\p_t N}{N}\,,\\\nonumber
\Gb_{tt}^{i} &= \frac{1}{2}\p_i(N^2) = N \p_i N\,,\\
\Gb_{ti}^{j} &= \frac{1}{2}g^{jk} \p_t g_{ik} = - N h_{ij} \,,\\\nonumber
\Gb_{ij}^{t} &= \frac{1}{2N^2}\p_t g_{ij} = -\frac{1}{N} h_{ij}\,,\\\nonumber
\Gb_{ij}^{k} &= \Gamma_{ij}^{k}\,,
\end{align}
where $\Gamma_{ij}^{k}$ are the Christoffel symbols of the metric $g$. In particular, we have 
$$
\Gb_{Ti}^j = - h_{ij}\quad\text{and}\quad \Gb_{ij}^T = -h_{ij}\,.
$$
Rewriting the last equations we have proved that
\begin{align}\label{eq-nw1s}
\D_{T} \W_{ijkl} &= -h_{il}\W_{TjTk}+h_{jl}\W_{TiTk}+h_{lp}\W_{ijkp} + h_{ik}\W_{TjTl}-h_{jk}\W_{TiTl}-h_{kp}\W_{ijlp} \\ \nonumber
&\quad +\frac12\left(\C_{ilT}g_{jk}+\C_{iTk}g_{jl}-\C_{jlT}g_{ik}+\C_{jTk}g_{il}\right) \,,
\end{align}
\begin{equation}\label{eq-nw2s}
\D_{l} \W_{ijkT} = -h_{il}\W_{TjTk}+h_{jl}\W_{TiTk}+h_{lp}\W_{ijkp} \,.
\end{equation}
On the other hand, by Lemma \ref{l-alg} one has
\begin{align*}
\D_T \W_{TiTj} &= \D_T \left(g^{kl} \W_{kilj}\right) \\
&= \left(\D_T \gb^{kl}\right) \W_{kilj} + g^{kl} \D_T \W_{kilj}\\
&= g^{kl} \D_T \W_{kilj} \\
&= g^{kl}\left( -h_{kj}\W_{TiTl}+h_{ij}\W_{TkTl}+h_{jp}\W_{kilp}+h_{kl}\W_{TiTj}-h_{il}\W_{TkTj}-h_{lp}\W_{kijp}\right)- \frac12 \C_{ijT}\\
&= H\W_{TiTj}-h_{il}\W_{TlTj}+h_{kp}\W_{kipj}- \frac12 \C_{ijT},
\end{align*}
i.e.
\begin{equation}\label{eq-nw3s}
\D_T \W_{TiTj} = H\W_{TiTj}-h_{il}\W_{TlTj}+h_{kp} \W_{kipj}- \frac12 \C_{ijT}\,.
\end{equation}
Moreover
\begin{align*}
h_{lp}\W_{ijkp} &= h_{lp}\left(\W_{TiTk}\,g_{jp}-\W_{TiTp}\,g_{jk}+\W_{TjTp}\,g_{ik}-\W_{TjTk}\,g_{ip}\right) \\
&= h_{jl}\W_{TiTk}-h_{il}\W_{TjTk}+h_{lp}\W_{TjTp}\,g_{ik}-h_{lp}\W_{TiTp}\,g_{jk}
\end{align*}
and
\begin{align}\label{eq-nonso}
h_{kp}\W_{kipj} &= h_{kp}\left(\W_{TkTp}\,g_{ij}-\W_{TkTj}\,g_{ip}+\W_{TiTj}\,g_{kp}-\W_{TiTp}\,g_{jk}\right) \\\nonumber
&= h_{kp}\W_{TkTp}\,g_{ij}-h_{ik}\W_{TjTk}-h_{jk}\W_{TiTk}+H\W_{TiTj}
\end{align}
Substituting in \eqref{eq-nw1s}, \eqref{eq-nw2s} and \eqref{eq-nw3s} (and using again Lemma \ref{l-alg}) we get
\begin{align*}
\D_{T} \W_{ijkl} &= 2\left( h_{ik} \W_{TjTl} - h_{jk} \W_{TiTl}+h_{jl} \W_{TiTk}-h_{il} \W_{TjTk}\right) \\
&\quad+h_{lp}\left(\W_{TjTp}\,g_{ik}-\W_{TiTp}\,g_{jk}\right) +h_{kp}\left(\W_{TiTp}\,g_{jl}-\W_{TjTp}\,g_{il}\right)\\
&\quad +\frac12\left(\C_{ilT}g_{jk}+\C_{iTk}g_{jl}-\C_{jlT}g_{ik}+\C_{jTk}g_{il}\right)\,,
\end{align*}
\begin{equation*}
\D_{l} \W_{ijkT} = 2\left(h_{jl} \W_{TiTk}-h_{il} \W_{TjTk}\right)+h_{lp}\left(\W_{TjTp}\,g_{ik}-\W_{TiTp}\,g_{jk}\right) \,,
\end{equation*}
\begin{equation*}
\D_T \W_{TiTj} = 2H \W_{TiTj} - 2h_{il}\W_{TjTl}-h_{jl}\W_{TiTl}+h_{kp}\W_{TkTp}\,g_{ij}- \frac12 \C_{ijT}\,.
\end{equation*}

\end{proof}

As a consequence, we obtain the following formula:
\begin{proposition}\label{p-firnv}  Let $(\X,\gb)$ be a globally hyperbolic spacetime satisfying \eqref{H1}. Then
$$
\frac12 \D_T |\W_{\a\b\g\d}|^2 = 16 H |\W_{TiTj}|^2-24h_{jl}\W_{TiTj}\W_{TiTl}-4\C_{ijT}\W_{TiTj}\,,
$$
where $\C$ is the Cotton tensor of $\gb$.
\end{proposition}
\begin{proof}
From Lemma \ref{l-for} one has
$$
\frac12 \D_T |\W_{\a\b\g\d}|^2 = 8 \left(\D_T \W_{TiTj} \right)\W_{TiTj}\,.
$$
Moreover
\begin{align*}
\D_T \W_{TiTj} &= \D_T \left(g^{kl} \W_{kilj}\right) = \left(\D_T \gb^{kl}\right) \W_{kilj} + g^{kl} \D_T \W_{kilj}\\
&= g^{kl} \D_T \W_{kilj}\\
&=g^{kl}\Big[2\left( h_{kl} \W_{TiTj} - h_{il} \W_{TkTj}+h_{ij} \W_{TkTl}-h_{kj} \W_{TiTl}\right) \\
&\quad+h_{jp}\left(\W_{TiTp}\,g_{kl}-\W_{TkTp}\,g_{il}\right) +h_{lp}\left(\W_{TkTp}\,g_{ij}-\W_{TiTp}\,g_{kj}\right)\\
&\quad +\frac12\left(\C_{kjT}g_{il}+\C_{kTl}g_{ij}-\C_{ljT}g_{ik}+\C_{iTl}g_{kj}\right)\Big]\\
&=2\left( H \W_{TiTj} - h_{ik} \W_{TkTj}-h_{kj} \W_{TiTk}\right) \\
&\quad+2h_{jp}\W_{TiTp}+h_{kp}\left(\W_{TkTp}\,g_{ij}-\W_{TiTp}\,g_{kj}\right)+\frac12\left(\C_{ijT}-\C_{ijT}+\C_{iTj}\right)\\
&=2 H \W_{TiTj} - 2h_{ik} \W_{TkTj}-h_{jp}\W_{TiTp}+h_{kp}\W_{TkTp}\,g_{ij}-\frac12\C_{ijT}\,.
\end{align*}
where we have used Lemma \ref{l-gra}. Thus
$$
\left(\D_T \W_{TiTj} \right)\W_{TiTj}=2 H |\W_{TiTj}|^2 - 3h_{jl}\W_{TiTj}\W_{TiTl}-\frac12\C_{ijT}\W_{TiTj}
$$
and the thesis follows.
\end{proof}

\

\

\bibliographystyle{abbrv}

\bibliography{biblio}

\

\

\end{document}